\documentclass[12pt,reqno]{amsart}

\usepackage{amsfonts}
\usepackage{amscd}
\usepackage{amssymb}
\usepackage{amsfonts}
\usepackage{amscd}
\usepackage{amsmath} 
\usepackage{mathrsfs}
\usepackage{dsfont}
\usepackage{enumerate}
\usepackage{xcolor}
\usepackage{float} % To use [H]
\usepackage[pdftex]{graphicx}
\allowdisplaybreaks
\usepackage{url}
\usepackage[normalem]{ulem}
\usepackage[hidelinks]{hyperref} 
 \hypersetup{
     colorlinks,
     linkcolor={red!60!black},
     citecolor={blue!80!black},
     urlcolor={blue!60!black}
 }

\date{\today}

\newtheorem*{theorem*}{Theorem}
\newtheorem{theorem}{Theorem}[section]
\newtheorem{corollary}[theorem]{\bf{Corollary}}
\newtheorem{lemma}[theorem]{Lemma}

\theoremstyle{definition}

\theoremstyle{remark}
\newtheorem{remark}[theorem]{\bf{Remark}}

\numberwithin{equation}{section}

\newcommand{\beas}{\begin{eqnarray*}}
\newcommand{\eeas}{\end{eqnarray*}}
\newcommand{\bes} {\begin{equation*}}
\newcommand{\ees} {\end{equation*}}
\newcommand{\be} {\begin{equation}}
\newcommand{\ee} {\end{equation}}
\newcommand{\bea} {\begin{eqnarray}}
\newcommand{\eea} {\end{eqnarray}}

\newcommand{\R}{\mathbb R}
\newcommand{\C}{\mathbb C}
\newcommand{\Z}{\mathbb Z}%
\newcommand{\N}{\mathbb N}

\newcommand{\X}{\mathbb{X}}

\renewcommand{\Im}{\text{Im}}
\renewcommand{\Re}{\text{Re}}

\renewcommand{\Re}{\operatorname{Re}}
\renewcommand{\Im}{\operatorname{Im}}
 
\usepackage{color}

\newcommand{\fa}{\mathfrak{a}}

\def \ee {\end{equation}}

\numberwithin{equation}{section}

\title [$L^p$-$L^q$ Fourier multipliers on Riemannian symmetric spaces]{$L^p$-$L^q$ Fourier multipliers and Hausdorff-Young-Paley inequalities on Riemannian symmetric spaces of noncompact type}

\author[Rana, Ruzhansky]{Tapendu Rana, Michael Ruzhansky}	
\address{Tapendu Rana  \endgraf Department of Mathematics: Analysis, Logic and Discrete Mathematics,	\endgraf Ghent University, 	\endgraf Krijgslaan 281, Building S8, B 9000 Ghent, Belgium.} \email{tapendurana@gmail.com, tapendu.rana@ugent.be}
\address{Michael Ruzhansky \endgraf Department of Mathematics: Analysis, Logic and Discrete Mathematics,	\endgraf Ghent University, 	\endgraf Krijgslaan 281, Building S8, B 9000 Ghent, Belgium. 
\endgraf and
\endgraf School of Mathematical Sciences
\endgraf Queen Mary University of London
\endgraf United Kingdom}
\email{michael.ruzhansky@ugent.be} 

\date{}

\begin{document}
\begin{abstract}
Our primary objective in this article is to establish H{\"o}rmander type $L^p \rightarrow L^q$ Fourier multiplier theorems in the context of noncompact type Riemannian symmetric spaces  $\mathbb{X}$ of arbitrary rank for the range $1 < p \leq 2 \leq q < \infty$. As a consequence of the Fourier multiplier theorem, we also derive a spectral multiplier theorem on $\mathbb{X}$. We then apply this theorem to prove $L^p \rightarrow L^q$ boundedness for functions of the Laplace-Beltrami operator and to obtain embedding theorems and operator estimates for the potentials and heat semigroups. Additionally, we provide mixed-norm versions of the Hausdorff-Young and Paley inequalities.  In this context, where the Fourier transform is holomorphic, and its domain consists of various strips, we present two versions of these inequalities and explore their interrelation. Furthermore, our findings and methods are also applicable to harmonic $NA$ groups, also known as Damek-Ricci spaces.
\end{abstract}
\maketitle
\tableofcontents
\addtocontents{toc}{\setcounter{tocdepth}{2}}
\section{Introduction}
The theory of translation-invariant bounded operators between $L^p$ spaces is a fundamental aspect of harmonic analysis, partly due to its applications in fields such as partial differential equations. In their seminal papers,  H\"ormander-Mikhlin \cite{Mih56, Hor60}  studied the boundedness of translation-invariant operators in the Euclidean setting. These operators can be characterized using the classical Euclidean Fourier transform on $\mathbb{R}^N$ and are thus known as Fourier multipliers.
In the context of $L^p \rightarrow L^q$ boundedness, it is known that one must have $p \leq q$. However, there is a philosophical difference between the cases where $p$ and $q$ are separated by $2$ and where they are not.   H\"{o}rmander \cite{Hor60} observed that, for the range $1 < p \leq 2 \leq q < \infty$, $L^p \rightarrow L^q$ estimates of multiplier operators do not rely on the regularity of the symbol.  This contrasts with other cases where the smoothness of the symbol is crucial.

Our main interest in this article is to establish H{\"o}rmander type $L^p \rightarrow L^q$ multiplier theorems for Fourier multiplier operators in the setting of Riemannian symmetric spaces (which we will refer to as symmetric spaces) of noncompact type for the range $1 < p \leq 2 \leq q < \infty$. These results apply to general Fourier multipliers and, in particular, yield new findings for spectral multipliers. We present several applications of these results, including $L^p \rightarrow L^q$ estimates for functions of the Laplace-Beltrami operator, embedding theorems, and estimates for heat semigroups. Furthermore, our results for spectral multiplier operators appear effective, encompassing known results expressed in terms of symbols as corollaries.

The question of $L^p$ boundedness of operators defined by multipliers that satisfy suitable symbol estimates has been extensively studied, with H{\"o}rmander-Mikhlin \cite{Mih56, Hor60} or Marcinkiewicz-type theorems providing a wealth of results. To mention a few, we refer the reader to \cite{Tha87, An90, DPW13, RW15} for $L^p$ multiplier results in various contexts. Additionally, analogs of these results in other noncommutative settings are established in \cite{GJP17, PRDL22, CGPT24}. In this article, we focus on $L^p \rightarrow L^q$ multipliers. A characteristic feature of $L^p \rightarrow L^q$ multipliers is that they require less regularity on the symbol. To provide context, we recall that in \cite{Hor60}, H{\"o}rmander initiated the study of $L^p \rightarrow L^q$ Fourier multipliers in the Euclidean setting. Specifically, if $B$ is a translation-invariant bounded operator on $L^2(\mathbb{R}^N)$ with Fourier transform $m_B$ of its convolution kernel, H{\"o}rmander established the following result concerning the $L^p \rightarrow L^q$ boundedness of the multiplier operator $B$.
\begin{theorem}[{{\cite[Theorem 1.11]{Hor60}}}] \label{thm_hor_Lp,Lq}
    Let $ 1<p\leq 2\leq q<\infty$. Assume that
     \begin{equation}
       \sup_{\alpha >0} \alpha \left( \int\limits_{\left\{ \xi \in \R^N : |m_B(\xi)|>\alpha \right\} } d\xi \right)^{\left(\frac{1}{p}-\frac{1}{q} \right)} <\infty.
    \end{equation}
    Then the operator $B$ is bounded from $L^p(\R^N)$ to $L^q(\R^N)$.    
\end{theorem}
The $L^p \rightarrow L^q$ boundedness of Fourier multipliers has recently been explored by researchers in various settings. Some works include studies on compact groups \cite{ANR16, KR22}, locally compact groups \cite{Cow75, AR20}, the generalized Fourier transform \cite{KR23}, and homogeneous trees \cite{CMS99, CMW19}. The second author with his collaborators investigated $L^p \rightarrow L^q$ Fourier multipliers for the Jacobi transform  \cite[Theorem 5.6]{KR21}, while in \cite{KRR24}, the analogue of this problem was studied for all Fourier matrix coefficients of $\mathrm{SL}(2,\R)$. Other relevant studies in this direction can be found in \cite{CK23, MRR24},  as well as the references therein. We now proceed by providing a specific description of the problem considered in this article to facilitate a mathematically rigorous discussion.

\subsection{H\"ormander's theorem on noncompact type symmetric spaces} In this article, we will explore the $L^p \rightarrow L^q$ boundedness properties of Fourier multiplier operators in the context of Riemannian symmetric spaces $\X = G/K$, where $G$ is a noncompact connected semisimple Lie group with finite center, and $K$ is a maximal compact subgroup of $G$.  Additionally, our results and techniques are applicable to harmonic $NA$ groups, also known as Damek-Ricci spaces. We refer the reader to Section \ref{sec_Preliminaries} for any other unexplained notation. 
It is well known that $G$-invariant operators $B$ that are bounded on $L^2(\X)$ are characterized by Fourier multiplier operators $T_m$, where $m$ is a bounded Weyl-invariant function on $\mathfrak{a}^{*}$ and
\begin{align}\label{def_multop_intB}
\widetilde{T_m f}(\lambda, k) := m (\lambda) \widetilde{f}(\lambda, k) \qquad \lambda \in \mathfrak{a}^*, k\in K,
\end{align}
 with $\widetilde{f}$ denoting the Fourier transform of the function $f: \X \rightarrow \mathbb{C}$ (see \cite[Prop. 1.7.1 and Ch. 6.1]{GV88} for more details). 
In 1990, Anker, in his remarkable work \cite{An90}, improved and generalized the previous results of Clerc and Stein \cite{CS71}, as well as those of other researchers \cite{ST78, AL86, Tay89}, by proving an analogue of the Hörmander-Mikhlin $L^p$ Fourier multiplier theorem on  Riemannian symmetric spaces of noncompact type. Later, Ionescu \cite{Ion02, Ion03} made further improvements to Anker's result. Ionescu's results are currently the best-known sufficient conditions of the Hörmander-Mikhlin type on $L^p(\X)$ when $p>1$. For the weak type $(1,1)$ case, the sharpest known results are provided by Meda and Vallarino \cite{MV10}. Recently, Wróbel \cite{Wro23} presented a multiplier theorem for rank-one symmetric spaces, which improves upon the results of both \cite{ST78} and \cite{Ion02}; see also \cite{RW20, MW21, PR23}.

Despite the outstanding contributions mentioned above, a general theory for the $L^p \rightarrow L^q$ boundedness of Fourier or spectral multiplier theorems on symmetric spaces of noncompact type has yet to be established. However, some $L^p \rightarrow L^q$ results are available for specific functions of the Laplace-Beltrami operator (see, for example, \cite{Lo89, An92, CGM93}). More recently, Anker et al. \cite[Appendix C]{AGL23} extended the result of \cite{CS71} by providing a necessary condition on the multiplier $m$ for the operator $T_m$ to be bounded from $L^p$ to $L^q$ when $p$ and $q$ are not separated by 2.  In this work, we present an analogue of the H{\"o}rmander type $L^p \rightarrow L^q$ multiplier theorem for $1 < p \leq 2 \leq q < \infty$ on general symmetric spaces of noncompact type.
\begin{theorem}\label{thm_Lp-Lq_int}
     Let $\X$ be a general symmetric space of noncompact type and let $1<p\leq 2 \leq q<\infty$. Then we have the following estimate:
     \begin{align}\label{eqn_T_m_Lp-Lq}
          \|T_m f\|_{L^q(\X)}& \leq C_{p,q} \left( \sup\limits_{\alpha>0}\,\alpha \left( \int\limits_{\{ \lambda \in \fa^* : |m(\lambda)|>\alpha \}   }  |c(\lambda)|^{-2} d\lambda\right)^{{\frac{1}{p}-\frac{1}{q}}} \right) \|f\|_{L^p(\X)}
     \end{align}
     for all $f \in L^p(\X)$,  where  $|c(\lambda)|^{-2} d\lambda$ is the Plancherel measure on $\X$.
\end{theorem}
The estimate \eqref{eqn_T_m_Lp-Lq} says that if the supremum on the right-hand side is finite, then the Fourier multiplier operator $T_m$ is bounded from $L^p(\X)$ to $L^q(\X)$. Furthermore, the theorem establishes the connection between the $L^p \rightarrow L^q$ norm of multiplier operators for $1 < p \leq 2 \leq q < \infty$ and Harish-Chandra's $c$-function. Next, we will demonstrate the application of the Fourier multiplier theorem in the important case of spectral multipliers.

Let $-\Delta$ denote the positive Laplace-Beltrami operator on $\X$, and let $|\rho|^2$ represent the bottom of its $L^2$-spectrum. By the Plancherel theorem, for a measurable function $\varphi$ on $[|\rho|^2, \infty)$, the spectral multiplier operator $\varphi(-\Delta)$ can be defined by
\begin{align}\label{def_spec_m}
\widetilde{\varphi(-\Delta) f}(\lambda,k) = \varphi(|\lambda|^2 + |\rho|^2) \widetilde{f}(\lambda, k), \quad \lambda \in \mathfrak{a}^*, k \in K,
\end{align}
for which we obtain the following result.
\begin{theorem}\label{thm_spec_mul_int}
 Let $\X$ be a general symmetric space of noncompact type and let $1<p\leq 2 \leq q<\infty$. Assume that $\varphi: (|\rho|^2, \infty)  \rightarrow \C$ is a function such that $|\varphi|$ decreases monotonically and continuously on $(|\rho|^2, \infty)$, with $\lim_{s\rightarrow \infty} |\varphi(s)| = 0$.   Then the spectral multiplier operator $\varphi(-\Delta)$  defined by \eqref{def_spec_m} satisfies the following estimates
      \begin{equation}\label{est_spec_op}
      \begin{aligned}
        \|\varphi(-\Delta)\|_{L^p(\X) \rightarrow L^q(\X)}  & \leq C_{p,q} \max \begin{cases} \sup\limits_{|\rho|^2 < s\leq 1+|\rho|^2}  |\varphi(s)| 
              \left( s-  |\rho|^2 \right)^{ \frac{\nu}{2}(\frac{1}{p}-\frac{1}{q})},\\
             \sup\limits_{s\geq 1+ |\rho|^2}  |\varphi(s)|  \left( s- |\rho|^2\right)^{\frac{n}{2}(\frac{1}{p}-\frac{1}{q})} ,
          \end{cases}
   \end{aligned}
   \end{equation}
   where  $ \|\varphi(-\Delta)\|_{L^p(\X) \rightarrow L^q(\X) }$ denotes the operator norm of $\varphi(-\Delta)$ from $L^p(\X)$ to $L^q (\X)$ and $n, \nu$ denote the topological dimension and pseudo-dimension of $\X$.
\end{theorem}

% \begin{theorem}\label{thm_spec_mul_int}
%  Let $\X$ be a general symmetric space of noncompact type and let $1<p\leq 2 \leq q<\infty$. Assume that $\varphi: (|\rho|^2, \infty)  \rightarrow \C$ is a function such that $|\varphi|$ decreases monotonically and continuously on $(|\rho|^2, \infty)$, with $\lim_{s\rightarrow \infty} |\varphi(s)| = 0$. Then the spectral multiplier operator $\varphi(-\Delta)$  defined by \eqref{def_spec_m} satisfies the following estimates
%       \begin{equation}\label{est_spec_op}
%       \begin{aligned}
%         \|\varphi(-\Delta)\|_{L^p(\X) \rightarrow L^q(\X)}  & \leq C_{p,q} \max \begin{cases} \sup\limits_{|\rho|^2 < s\leq 1+|\rho|^2}  |\varphi(s)| 
%               \left( s-  |\rho|^2 \right)^{ \frac{\nu}{2}(\frac{1}{p}-\frac{1}{q})},\\
%              \sup\limits_{s\geq 1+ |\rho|^2}  |\varphi(s)|  \left( s- |\rho|^2\right)^{\frac{n}{2}(\frac{1}{p}-\frac{1}{q})} ,
%           \end{cases}
%    \end{aligned}
%    \end{equation}
%    where  $ \|\varphi(-\Delta)\|_{L^p(\X) \rightarrow L^q(\X) }$ denotes the operator norm of $\varphi(-\Delta)$ from $L^p(\X)$ to $L^q (\X)$ and $n, \nu$ denote the topological dimension and pseudo-dimension of $\X$.
% \end{theorem}
Interestingly, this result asserts that the $L^p \rightarrow L^q$ norm of the spectral multiplier operator $\varphi(-\Delta)$ depends primarily on the behavior near $|\rho|^2$ and the growth rate at infinity of the function $s \rightarrow |\varphi(s)|$. However, we would like to note that if $\varphi$ satisfies the hypotheses of Theorem \ref{thm_spec_mul_int}  and $\lim_{s\rightarrow {|\rho|^2}{+}} |\varphi(s )|$ is finite, or in particular, if  $|\varphi|$  extends continuously on $[|\rho|^2, \infty)$, then the supremum in the first line of \eqref{est_spec_op} is always finite. In that case, the statement could be simplified by only retaining the supremum in the second line of \eqref{est_spec_op}, which governs the $L^p \rightarrow L^q$  boundedness. %Nevertheless, we present it in the form \eqref{est_spec_op} so that $C_{p,q}$  remains independent of $\varphi$.

Moreover, the estimate for the operator norm of $\varphi(-\Delta)$ can be applied to functions of the Laplace-Beltrami operator on $\X$. Specifically, we derive the $L^p \rightarrow L^q$ estimates for operators of the form \begin{align}\label{defn_Rz^s}
    \mathcal{R}_z^{\sigma}:=(zI - \Delta)^{-\frac{\sigma}{2}}
    \end{align} 
    by applying Theorem \ref{thm_spec_mul_int} with $\varphi(s) = (z+s)^{-\frac{\sigma}{2}}$ and $1<p\leq 2 \leq q<\infty$. We consider $z \in \mathbb{C}$ with $\Re z \geq -|\rho|^2$ and $\sigma \in \C$ with $\Re \sigma \geq 0$ to ensure that the operators $\mathcal{R}_z^{\sigma}$ are bounded on $L^2(\X)$. The $L^p \rightarrow L^q$ boundedness properties of operators of the form $\mathcal{R}_z^{\sigma}$, for suitable values of $z$ and $\sigma$, have been extensively studied in various settings (see \cite[Sec. 4]{An92} for a detailed discussion and appropriate references).

In the context of noncompact type symmetric spaces, Stanton and Tomas \cite[Sec. 6]{ST78} initiated the study of the boundedness of the operator $ (-\Delta)^{-\frac{\sigma}{2}}$ in the rank-one setting (see also \cite{LR82, DST88}). Anker and Lohou\'e extended these rank-one results to symmetric spaces with real form \cite{AL86}. For general symmetric spaces of noncompact type, the first result in this direction is due to Varopoulos \cite{Var88}, who studied the $L^p \rightarrow L^q$ boundedness of the operator $(-\Delta -|\rho|^2)^{-\frac{1}{2}}$. Lohou\'e \cite{Lo89}  further generalized Varopoulos's results to a broader family of operators. Specifically, Lohou\'e proved that the family of operators \begin{align*}
    (-\Delta -|\rho|^2)^{-\frac{\sigma}{2}}, \quad \sigma \in \C \text{ with } \Re \sigma \geq 0 
\end{align*} maps $L^p(\X)$ to $L^q(\X)$ only if $p \leq 2 \leq q$. Additionally, he showed that if $1 < p \leq 2 \leq q < \infty$ and $\Re \sigma \geq n(1/p - 1/q)$, then $(-\Delta -|\rho|^2)^{-\frac{\sigma}{2}}$ is bounded from $L^p(\X)$ to $L^q(\X)$. Subsequently, Taylor \cite{Tay89} and Anker \cite{An90, An92} provided a nearly complete description of the $L^p \rightarrow L^q$ boundedness for the entire family of operators $\mathcal{R}_{z}^{\sigma}$ for $ \Re z>-|\rho|^{2}$, excluding certain critical indices that were finally addressed in \cite{CGM93}. Recently, in \cite{KRZ_23} (see also \cite{KKR24}), the second author with his collaborators studied the operators $\mathcal{R}_{\zeta}^{\sigma}$ for $\zeta, \sigma  \in \mathbb{R}$ with $\zeta \geq -|\rho|^{-2}, \sigma\geq 0
$, thoroughly examining the necessary aspects as well.

These aforementioned results rely on quite explicit knowledge of the kernels of the operators involved. Our approach, however, is quite different. As a consequence of Theorem \ref{thm_spec_mul_int}, we derive the following result. Additionally, we have obtained estimates for the norms of the operators $\mathcal{R}_{z}^{\sigma}$ (in Section \ref{subsec_potential}).
\begin{corollary}\label{cor_potentials}
Let $1 < p \leq 2 \leq q < \infty$. The operator $\mathcal{R}_{z}^{\sigma}$, defined in \eqref{defn_Rz^s}, is bounded from $L^p(\X)$ to $L^q(\X)$ if
\begin{align*}
\Re z > -|\rho|^2 \quad \text{and} \quad \Re \sigma \geq n \left( \frac{1}{p} - \frac{1}{q} \right).
\end{align*}
 Furthermore, for the limit case $\Re z = -|\rho|^2$, the statement above holds true if $\Im z \neq 0$.  If in addition $\Im z = 0$, that is $z =-|\rho|^2$, the same result is true provided that
   \begin{align*}
       \nu \left( \frac{1}{p} -\frac{1}{q}\right)   \geq \Re \sigma \geq n \left( \frac{1}{p} -\frac{1}{q}\right).
   \end{align*}
\end{corollary}
By performing a similar calculation as in Corollary \ref{cor_potentials}, we derive the following Sobolev-type embedding on $\X$. For $\sigma>0$ and $1<p<\infty$, we define the Sobolev space $H^{\sigma,p}(\X)$ as the image of $L^p(\X)$ under the operator $(-\Delta)^{\frac{\sigma}{2}}$, equipped with the norm $\| f\|_{H^{\sigma,p}(\X)} := \|(-\Delta)^{\frac{\sigma}{2}} f\|_{L^p(\X)} .$
\begin{corollary}[Sobolev type embedding]\label{cor_sobolev}
    Let $1<p\leq 2 \leq q< \infty$ and $\sigma,\kappa \in \R$. Then for  $\sigma-\kappa \geq  {n(\frac{1}{p}-\frac{1}{q})}  $, there exists a constant $C=C(\sigma, \kappa,p,q)>0$ such that  for all $ f\in C_c^{\infty}(\X)$
    \begin{align*}
        \left\| (-\Delta)^{\frac{\kappa}{2}} f\right\|_{L^q(\X)} \leq C_{p,q}   \left\| (-\Delta)^{\frac{\sigma}{2}} f\right\|_{L^p(\X)}.
    \end{align*}
    In particular, by taking $\kappa=0$ in the inequality above, we obtain the following Sobolev-type embedding for $\sigma \geq  {n(\frac{1}{p}-\frac{1}{q})}$
    \begin{align}\label{eqn_sob_emb}
    \left\|  f\right\|_{L^q(\X)}  \leq   C_{p,q}   \|f\|_{H^{\sigma, p}(\X)}.  
    \end{align}
\end{corollary}
We like to mention that, in contrast to the Euclidean analogue of \eqref{eqn_sob_emb}, the set of admissible pairs $(1/p, 1/q)$ that satisfy the Sobolev inequality \eqref{eqn_sob_emb} is much larger. In the context of $\R^N$, the balance condition $1/p - 1/q = \sigma/N$ with $\sigma < N$ is always necessary, see \cite{Bec95, Bec08, DCGT17}. 
 Furthermore, in noncompact type symmetric spaces, it is known that the Sobolev inequality \eqref{eqn_sob_emb} holds for $\sigma > 0$ and $1 < p \leq q < \infty$ if and only if $\sigma \geq n(1/p -1/q)$ (see for instance \cite[(1.13)]{KRZ_23}).

The estimate for the heat semigroup has been intensively studied by many authors. For a comprehensive overview, we refer to \cite{Ank88} and \cite[Sec. 3]{An92, CGM93}. We will conclude our applications by deriving an upper bound for the heat semigroup using Theorem \ref{thm_spec_mul_int}. Let us consider the following $\sigma$-fractional heat equation for $ \sigma >0$ and $1 < p \leq 2$:
\begin{equation}
\begin{aligned}\label{eqn_frac heat}
\frac{\partial}{\partial t}u(x,t) + (-\Delta)^{\sigma} u(x,t)=0 , \quad t > 0, \quad u(0,x) = u_0(x) \in L^p(\X).
\end{aligned}
\end{equation}
Using Theorem \ref{thm_spec_mul_int}, we derive an upper bound for the solution operator $u(t, x) = e^{-t (-\Delta)^{\sigma}}u_0(x)$.
\begin{corollary}\label{cor_est_fr_heat}
Let $1<p\leq 2\leq q <\infty$ and $\sigma >0$. Then we have shown that
   \begin{align*}
        \| e^{-t (-\Delta)^{\sigma}} \|_{L^p(\X) \rightarrow L^q(\X)} \lesssim \begin{cases} t^{-\frac{n}{2\sigma }(\frac{1}{p}-\frac{1}{q})} \quad &\text{if }0<t \leq 1 ,  \\
             t^{-\frac{\nu}{2} (\frac{1}{p}-\frac{1}{q})} e^{- t|\rho|^{2\sigma}} \quad & \text{if } t\geq 1.
        \end{cases}
   \end{align*}
\end{corollary}
\begin{remark}
We would like to mention that the estimates in the corollary above for the fractional heat semigroup when $t$ is near zero are sharp and coincide with those in \cite{CGM93, CGM95}, where precise estimates of $ \| e^{-t (-\Delta)^{\sigma}} \|_{L^p(\X) \rightarrow L^q(\X)}$ were obtained. However, for $t$ approaching infinity, we did not achieve the sharpest estimates. Specifically, the authors in \cite[Theorem 3.2]{CGM93} and \cite[Theorem 1]{CGM95} provided sharp estimates for fractional heat semigroups for all $1 \leq p \leq q \leq \infty$. Comparing our results with theirs, we note that for $1 < p \leq 2 \leq  q < \infty$ and large $t$, the exponential decay parts are identical, but the polynomial terms obtained by the authors in \cite{CGM93, CGM95} are sharper.
\end{remark}
\subsection{Hausdorff-Young-Paley inequalities on noncompact type symmetric spaces} A critical aspect of our study involves obtaining analogues of  Hausdorff-Young and Paley inequalities in these non-Euclidean settings. Unlike in Euclidean settings, the (Poisson) kernel used to define the (Helgason) Fourier transform $\widetilde{f}$ on symmetric spaces is not bounded. This necessitates modifications to the Euclidean results and naturally suggests the use of restriction inequalities, leading to the following version of Hausdorff-Young inequality and its dual inequality on general symmetric spaces of noncompact type. For $1 \leq p \leq \infty$, let $p' = p/(p-1)$ denotes the conjugate exponent of $p$. When $p=\infty$, by $(\int |F|^{p})^{\frac{1}{p}}$ we will represent the $L^{\infty}$ norm of $F$.
\begin{theorem}\label{thm_HY_UD}
     Let $\X$ be a general symmetric space of noncompact type and let $1\leq p\leq 2 $. Then we have for all $f\in C_c^{\infty}(\X)$
\begin{align}\label{HY_ud_int}
    \left( \frac{1}{|W|}\int_{\mathbb{\fa^*}} \left(\int_K|\widetilde{f}(\lambda, k)|^{2} \, dk\right)^{\frac{p'}{2}} |c(\lambda)|^{-2} \,d \lambda \right)^{\frac{1}{p'}} \leq  \|f\|_{L^p(\X)}.
\end{align}    
Furthermore, if $p \geq 2$, then we have for all $f\in C_c^{\infty}(\X)$
\begin{align}\label{eqn_HYP_dual-int}
     \|f\|_{L^p(\X)} \leq   \left( \frac{1}{|W|}\int_{\mathbb{\fa^*}} \left(\int_K|\widetilde{f}(\lambda, k)|^{2} \, dk\right)^{\frac{p'}{2}} |c(\lambda)|^{-2} \,d \lambda \right)^{\frac{1}{p'}} .
\end{align}
\end{theorem}

The Hausdorff-Young inequality, in its classical form, asserts that for $1 \leq p \leq 2$, the Fourier transform maps functions from $L^p$ space to $L^{p'}$ space with respect to Plancherel measure.  Paley's inequality addresses whether, for a given $p \in [1, 2)$, there exists a measure $\mu$ such that the Fourier transform operator is bounded from $L^p$ to $L^p(\mu)$. H{\"o}rmander \cite[Theorem 1.10]{Hor60} answered this affirmatively by proving Paley's inequality on $\mathbb{R}^N$. We offer the following analogue of Paley's inequality in the context of noncompact type symmetric spaces, which can also be seen as a weighted version of Plancherel’s formula for $L^p(\X)$ with $1 < p \leq 2$.
\begin{theorem}\label{thm_paley_uni}  Let $\X$ be a general symmetric space of noncompact type and  let $1<p\leq 2$.
 Assume that $u$ is a positive function on $\fa^*$ satisfying the following condition 
	\begin{align}\label{cond_u_HYP}
	 	 \|u\|_{c,\infty} := \sup_{\alpha>0} \alpha  \int\limits_{\substack{\{\lambda \in \fa^* : u(\lambda)>\alpha\} }} |c(\lambda)|^{-2} \,d\lambda <\infty.
	 \end{align}
  Then we have for all $f\in C_c^{\infty}(\X)$
  \begin{align}\label{eqn_paley}
       \left( \int_{\mathbb{\fa^*}} \left(\int_K|\widetilde{f}(\lambda, k)|^{2} \, dk\right)^{\frac{p}{2}} u(\lambda)^{2-p} |c(\lambda)|^{-2} \,d \lambda \right)^{\frac{1}{p}} \leq C_p  \|u\|_{c,\infty}^{\frac{2}{p}-1}\|f\|_{L^p(\X)}.
  \end{align}
\end{theorem}
In the theorems above, we consider the Fourier inequalities by treating the Fourier transform $\widetilde{f}$ as a function on $\mathfrak{a}^* \times K$. However, the holomorphic extension of the Fourier transform of an $L^p(\X)$ function $(1 \leq p < 2)$ is perhaps the most distinctive feature of noncompact type symmetric spaces. Recognizing this characteristic, Ray and Sarkar \cite{RS_09} used Stein's analytic interpolation theory to provide an appropriate version of the Hausdorff-Young inequality using the Plancherel measure for harmonic $NA$ groups (particularly in rank-one symmetric spaces of noncompact type). However, their version of the Hausdorff-Young inequality may not be useful for \textit{reinterpolation analytically} with other inequalities. More generally, using Stein's interpolation theory for a family of analytic linear operators $\{ T_z : 0 \leq \Re z \leq 1 \}$, we may obtain a norm estimate for $\{ T_{\theta}: 0 \leq \theta \leq 1\}$ but not for $\{T_{\theta+i\xi}: 0 \leq \theta \leq 1, \xi \in \R\}$, which is required for \textit{reinterpolation analytically}.

In the following, we provide a version of the Hausdorff-Young inequality using the Stein-Weiss analytic interpolation theorem that can be interpolated with other inequalities to obtain new results. Specifically, we interpolate this Hausdorff-Young inequality with the Paley inequality (Theorem \ref{thm_paley_int}) to derive Hausdorff-Young-Paley inequalities on non-unitary duals in the context of rank-one symmetric spaces of noncompact type. Additionally, we investigated the relationship between our version of the Hausdorff-Young inequality and that by Ray and Sarkar (Theorem \ref{thm_RS_HYinq}). We have proved the following version of the Hausdorff-Young inequality on non-unitary duals.
\begin{theorem}\label{thm_H-Y_int}
    Let $\X$ be a rank one symmetric space of noncompact type, $1\leq p\leq 2$, and $\rho_p := (2/p-1)\rho$. Then for any $p\leq q\leq  p'$, we have
     \begin{align*} 
        \left( \int_{\mathfrak{a}^*} \left( \int_K  \left|\widetilde {f}(\lambda +i\rho_q , k) \right|^q  \frac{ |\lambda+i\rho_q|^q}{|\lambda+i\rho_q +i 2\rho|^q}  \,  dk \right)^{\frac{p'}{q}}  (1+|\lambda|)^{n-1} d\lambda \right) ^{\frac{1}{p'}} \leq C_{p,q}  \|f\|_{L^p(\X)},
    \end{align*}
    for all $f\in C_c^{\infty}(\X)$, where $n\geq 1$ is the dimension of $\X$.
\end{theorem}
For rank one case, we established the following Paley-type inequality on non-unitary duals.
\begin{theorem}\label{thm_paley_int}
    Let $ \X$ be a rank one symmetric space of noncompact type and let  $1\leq p\leq 2$. Assume that $u>0$ is a measurable function on  $\fa^*$ such that  
    \begin{align}\label{eqn_int_u}
    \|u\|_{L^1(|c(\lambda)|^{-2}) } :=  \int_{\fa^*}  u(\lambda)  |c(\lambda)|^{-2}  d\lambda <\infty.
    \end{align}
 Then for any $p\leq q\leq  p'$, we have
     \begin{align*}
        &\left( \int_{\fa^*} \left( \int_K  \left|\widetilde {f}(\lambda +i\rho_{q} , k) \right|^{q}  \frac{ |\lambda+i\rho_{q}|^{q}}{|\lambda+i\rho_{q} +i 2\rho|^{q}}  \,  dk \right)^{\frac{p}{q}}  u(\lambda)^{2-p} |c(\lambda)|^{-2} d\lambda \right) ^{\frac{1}{p}} \\
        & \hspace{10cm}\leq C_{p,q} \|u\|_{L^1(|c(\lambda)|^{-2})}^{{\frac{2}{p}-1}} \|f\|_{L^p(\X)},
    \end{align*}
    for all $f\in C_c^{\infty}(\X)$.
\end{theorem}
Finally, by applying the Stein-Weiss analytic interpolation theory for mixed norm spaces and analytically interpolating the Hausdorff–Young inequality (Theorem \ref{thm_H-Y_int}) with the Paley inequality (Theorem \ref{thm_paley_int}), we establish the following Hausdorff-Young-Paley inequality on non-unitary duals.
 \begin{theorem}\label{thm_H-Y-P} 
 Let $\X$ be a rank one symmetric space of noncompact type, $1< p\leq 2$, and $p\leq b \leq p'$. Assume that $u>0$ is a measurable function on  $\fa^*$ such that \eqref{eqn_int_u} holds. 
    Then for any $p\leq q_0, q_1\leq p'$, we have the following estimate
     \begin{align*}
        \left( \int_{\mathfrak{a}^*} \left( \int_K  \left|\widetilde {f}(\lambda +i\rho_{q_{\theta}} , k) \right|^{q_\theta}  \frac{ |\lambda+i\rho_{q_{\theta}}|^{q_{\theta}}}{|\lambda+i\rho_{q_{\theta}} +i 2\rho|^{q_{\theta}}}  \,  dk \right)^{\frac{b}{q_{\theta}}}  u(\lambda)^{1-\frac{b}{p'}} |\lambda|^{  \frac{2(p'-b)}{p'(2-p)}} (1+|\lambda|)^{n-3+\frac{2(b-p)}{p'(2-p)}}  d\lambda \right) ^{\frac{1}{b}} \\
         \leq C_{p,q_{\theta}} {\|u\|_{L^{1}(|c(\lambda)|^{-2})}^{\frac{1}{b}-\frac{1}{p'}}} \|f\|_{L^p(\X)},
    \end{align*}
    for all $f\in C_c^{\infty}(\X)$, where 
    \begin{align*}
    \frac{1}{b}= \frac{1-\theta}{p}+\frac{\theta}{p'} \qquad \text{and} \qquad  \frac{1}{q_{\theta}}= \frac{1-\theta}{q_0}+\frac{\theta}{q_1}.
    \end{align*}
 \end{theorem}
\begin{remark}\label{rem_HYP_fix}
\begin{enumerate}
    \item  Taking $q_0 =q_1$,  we get the following version of the Hausdorff-Young-Paley inequality from the theorem above on a fixed line, that is for $1<p\leq 2$ and any $p\leq b, q \leq p'$, we have 
\begin{equation}\label{eqn_HYP_nP}
     \begin{aligned}
        \left( \int_{\fa^*} \left( \int_K  \left|\widetilde {f}(\lambda +i\rho_{q} , k) \right|^{q}  \frac{ |\lambda+i\rho_{q}|^{q}}{|\lambda+i\rho_{q} +i 2\rho|^{q}}  \,  dk \right)^{\frac{b}{q}}  u(\lambda)^{1-\frac{b}{p'}} |\lambda|^{  \frac{2(p'-b)}{p'(2-p)}} (1+|\lambda|)^{n-3+\frac{2(b-p)}{p'(2-p)}}  d\lambda \right) ^{\frac{1}{b}} \\
         \leq C_{b,q} {\|u\|_{L^{1}(|c(\lambda)|^{-2})}^{\frac{1}{b}-\frac{1}{p'}}} \|f\|_{L^p(\X)},
    \end{aligned}
    \end{equation}
     \item By substituting $b=p'$ and $b=p$   into the inequality \eqref{eqn_HYP_nP}, it is reduced to Hausdorff-Young inequality (Theorem \ref{thm_H-Y_int}) and Paley's inequality (Theorem \ref{thm_paley_int}), respectively.
     \end{enumerate}
\end{remark}
\begin{corollary}\label{cor_of_HYP}
 Let $\X$ be a rank one symmetric space of noncompact type, $1<p<2$, and $p\leq b, q \leq p'$.  If $q\not =2$, then assuming \eqref{eqn_int_u} the subsequent Hausdorff-Young-Paley inequality with respect to the Placherel measure follows from  \eqref{eqn_HYP_nP}
     \begin{align}\label{eqn_HYP_P}
        \left( \int_{\fa^*} \left( \int_K  \left|\widetilde {f}(\lambda +i\rho_{q} , k) \right|^{q}    \,  dk \right)^{\frac{b}{q}}  u(\lambda)^{1-\frac{b}{p'}}  |c(\lambda)|^{-2} d\lambda \right) ^{\frac{1}{b}}
         \leq C_{b,q} {\|u\|_{L^{1}(|c(\lambda)|^{-2})}^{\frac{1}{b}-\frac{1}{p'}}} \|f\|_{L^p(\X)}.
    \end{align}
\end{corollary}
We conclude this section with an outline of the article. The next section provides essential background on harmonic analysis related to semisimple Lie groups and symmetric spaces of noncompact type, along with a review of relevant previous results. In Section \ref{sec_HYP_UD}, we establish the Hausdorff-Young and Paley inequalities for unitary duals in general symmetric spaces of noncompact type, specifically Theorems \ref{thm_HY_UD} and \ref{thm_paley_uni}. Section \ref{sec_Lp-Lq_multiplier} addresses the $L^p\rightarrow L^q$ boundedness of Fourier multipliers, presenting Theorems \ref{thm_Lp-Lq_int} and \ref{thm_spec_mul_int} and their applications. In Section \ref{sec_HYP_ND}, we prove the holomorphic versions of the Hausdorff-Young and Paley inequalities, specifically Theorems \ref{thm_H-Y_int} through \ref{thm_H-Y-P}. In Section \ref{sec_NA_group}, we outline how to extend our results to Harmonic $NA$ groups. Finally, in Section \ref{sec_SW_analytic}, we develop the theory of Stein-Weiss analytic interpolation, completing our study.
\section{{Preliminaries}} \label{sec_Preliminaries}
\subsection{Generalities}
The symbols $\N$, $\Z$, $\R$, and $\C$ will respectively denote the set of natural numbers, the ring of integers, and the fields of real and complex numbers. For $z \in \C$, we use the notations $\Re z$ and $\Im z$ to refer to the real and imaginary parts of $z$, respectively. We will adhere to the standard convention of using the letters $C$, $C_1$, $C_2$, etc., to represent positive constants whose values may vary from one line to another. Throughout this article, we use $X \lesssim Y$ or $X \gtrsim Y$ to indicate the estimates $X \leq CY$ or $X \geq CY$ for some absolute constant $C > 0$. Additionally, the notation $X \asymp Y$ will be used to signify that both $X \lesssim Y$ and $X \gtrsim Y$ hold. 
\subsection{Riemannian symmetric spaces}
Here, we review some general facts and necessary preliminaries regarding semisimple Lie groups and harmonic analysis on Riemannian symmetric spaces of noncompact type. Most of this information is already known and can be found in \cite{GV88, Helgason_GA}. We will use the standard notation from \cite{Ion03, BP22}. To keep the article self-contained, we will include only the results necessary for this article, following the preliminaries outlined in \cite{Ion03, BP22}.
   
    Let $G$ be a connected, noncompact, real semisimple Lie group with a finite center, and let $\mathfrak{g}$ be its Lie algebra. Let  $\theta$ be a Cartan involution of $\mathfrak{g}$, inducing the decomposition   $\mathfrak{g} = \mathfrak{k} \oplus \mathfrak{p}$. Let $K= \exp \mathfrak{k} $ be a maximal compact subgroup of $G$, and let $\mathbb{X} = G / K$ be an associated symmetric space with origin $\mathbf{0} = \{eK\}$. 

    Let $\mathfrak{a}$ be a maximal abelian subalgebra of $\mathfrak{p}$. The dimension of $\mathfrak{a}$ determines the rank of the symmetric space $\mathbb{X}$, denoted by $l$. The Killing form of $\mathfrak{g}$ induces an inner product $\left\langle \cdot, \cdot \right\rangle$ on $\mathfrak{a}$, allowing us to identify $\mathfrak{a}$ with $\mathbb{R}^l$. Let $\mathfrak{a}^*$ and $\mathfrak{a}_{\mathbb{C}}^*$ denote the real and complex duals of $\mathfrak{a}$, respectively. For any $\lambda \in \mathfrak{a}^*$, let $H_{\lambda}$ be the unique element of $\mathfrak{a}$ such that
    \begin{align*}
    \lambda(H) = \left\langle H_{\lambda}, H \right\rangle, \quad \text{for all } H \in \mathfrak{a}.
    \end{align*}
   With this, we transfer the inner product $\left\langle \cdot, \cdot \right\rangle$ from $\mathfrak{a}$ to $\mathfrak{a}^*$, and, by abuse of notation, denote it again by $\left\langle \cdot, \cdot \right\rangle$, defined by the rule
    \begin{align*}
    \left\langle \lambda, \lambda' \right\rangle= \left\langle H_{\lambda}, H_{\lambda'} \right\rangle, \quad \text{for all }  \lambda, \lambda' \in \mathfrak{a}^*.
\end{align*}
We further extend the inner product on $\mathfrak{a}^*$ to a $\mathbb{C}$-bilinear form on $\mathfrak{a}_{\mathbb{C}}^*$. This inner product allows us to identify $\mathfrak{a}^*$ with $\mathbb{R}^l$. In what follows, we will use the notation $|\cdot|$ to refer to both norms associated with the inner products on $\mathfrak{a}$ and $\mathfrak{a}^*$.

Let $\Sigma \subset \mathfrak{a}^*$ be the set of restricted roots for the pair $(\mathfrak{g}, \mathfrak{a})$. Let $M$ denote the centralizer of $\mathfrak{a}$ in $K$, and let $W$ represent the Weyl group associated with $\Sigma$. For each $\alpha \in \Sigma$, let $\mathfrak{g}_\alpha$ be the corresponding root space, with $m_\alpha = \dim \mathfrak{g}_\alpha$. We choose a system of positive roots $\Sigma^+$ and define the associated positive Weyl chamber as $\mathfrak{a}^+ = \{H \in \mathfrak{a} : \alpha(H) > 0 \text{ for all } \alpha \in \Sigma^+\}$. Let $\Sigma_0^+$ denote the set of positive indivisible roots. The half-sum of all positive roots, counted with their multiplicities, is given by
\begin{align*} 
\rho = \frac{1}{2} \sum_{\alpha \in \Sigma^+} m_\alpha \alpha. 
\end{align*}

% Let $\Sigma \subset \mathfrak{a}^* $ be the set of restricted roots for the pair $(\mathfrak{g}, \mathfrak{a})$, let $M$ denote the centralizer of $\mathfrak{a}$ in $K$, and let $W$ represent the Weyl group associated with $\Sigma$. For each $\alpha \in \Sigma$, let $\mathfrak{g}_\alpha$ be the corresponding root space, with $m_\alpha = \dim \mathfrak{g}_\alpha$. We choose a system of positive roots $\Sigma^+$ and define the corresponding positive Weyl chamber $\mathfrak{a}^+ = \{H \in \mathfrak{a} : \alpha(H) > 0 \text{ for all } \alpha \in \Sigma^+\}$. Let $\Sigma_0^+$ denote the set of positive indivisible roots. Let $\rho$ denote the half-sum of all positive roots, counted with their multiplicities: 
Let $n$ represent the dimension of $\mathbb{X}$ and $\nu$ the pseudo-dimension (or dimension at infinity) of $\mathbb{X}$. These quantities are given by the following expressions
% Denote by $n$ the dimension of $\mathbb{X}$ and $\nu$ the pseudo-dimension (or dimension at infinity) of $\mathbb{X}$. These quantities are given by the following expressions: 
\begin{align*} n = l + \sum_{\alpha \in \Sigma^+} m_\alpha, \qquad \nu = l + 2 | \Sigma_0^+ |. \end{align*} In general, we have $l \geq 1$, $n \geq 2$, and $\nu \geq 3$. 

We denote by $\mathfrak{n} = \bigoplus_{\alpha \in \Sigma^+} \mathfrak{g}_\alpha$ the nilpotent Lie subalgebra of $\mathfrak{g}$ associated with $\Sigma^+$, and by $N$ and $A$ the corresponding Lie subgroups of $\mathfrak{n}$ and $\mathfrak{a}$, respectively, in $G$. The Iwasawa decomposition of $G$ is given by $G = KAN$, meaning each element $g \in G$ can be uniquely expressed as \begin{align*} g = K(g) \exp(H(g))N(g), \quad K(g) \in K,\, H(g) \in \mathfrak{a}, \, N(g) \in N. \end{align*} Let $A^{+} = \exp(\mathfrak{a}^+)$, and let $\overline{A^{+}}$ denote the closure of $A^{+}$ in $G$. We then obtain the polar decomposition $G = K \overline{A^{+}} K$. Under this decomposition, the Haar measure on $G$ can be expressed as follows.

\begin{equation*}
\begin{aligned}
\int_G f(g) d g & =\int_K \int_{\mathfrak{a}^{+}} \int_K f\left(k_1 (\exp H) k_2\right) J(\exp H) \, d k_1\, d H\, d k_2
\end{aligned}
\end{equation*}
where  $J(\exp H)=C_0 \prod_{\alpha \in \Sigma^{+}}(\sinh \alpha(H))^{m_\alpha} $,  $C_0$ is a normalizing constant,  $K$ is equipped with its normalized Haar measure, and $dH$ is the Lebesgue measure on $\R^l$. If $f$ is a function on $\mathbb{X} = G / K$, it can be viewed as a right $K$-invariant function on $G$. Additionally, we recall the following integration formula on $K$ (see \cite[Lemma 5.19, p.197]{He_00}). 
\begin{lemma}\label{k(xk)_dif}
    Let $g \in G$. The mapping $T_g:k \mapsto K(gk)$ is a diffeomorphism of $K$ onto itself and 
    \begin{equation}\label{F_C(K)}
        \int_K F(K(gk)) \, dk =\int_K F(k) e^{-2 \rho H(g^{-1}k)} \, dk, \quad F\in C(K),
    \end{equation}
    where $K(gk)$ represents the unique component of $gk$ in the maximal compact subgroup $K$ from the Iwasawa decomposition.
\end{lemma}
The particular case $F\equiv 1$ in \eqref{F_C(K)}, will be useful for us:
 \begin{equation}\label{F_C(K)=1}
         \int_K   e^{-2 \rho H(g^{-1}k)} \, dk=1,\quad \text{for all }g \in G.
    \end{equation}
\subsection{Fourier transform on Riemannian symmetric spaces}
The Fourier transform $\widetilde{f}$ of a smooth, compactly supported function $f$ on $\mathbb{X}$ is defined on $\mathfrak{a}_{\mathbb{C}}^* \times K$ and is given by (see \cite[page 199]{Helgason_GA}):
\begin{equation}\label{defn:hft}
\widetilde{f}(\lambda,k) = \int_{G} f(g) e^{(i\lambda- \rho)H(g^{-1}k)} dg,\:\:\:\:\:\:  \lambda \in \mathfrak{a}_{\mathbb{C}}^*, k \in K,
\end{equation}
where $H(g^{-1} k)$ represents the unique $\mathfrak{a}$-component from the Iwasawa decomposition of $g^{-1}k$.  Since $M$ normalizes $N$, the function $k \mapsto \widetilde{f}(\lambda, k)$ is right $M$-invariant. It is known that if $f \in L^1(\mathbb{X})$, then $\widetilde{f}(\lambda, k)$ is continuous with respect to $\lambda \in \mathfrak{a}^*$ for almost every $k \in K$ (and is, in fact, holomorphic in $\lambda$ on a domain containing $\mathfrak{a}^*$). Furthermore, if $\widetilde{f} \in L^1(\mathfrak{a}^* \times K, |c(\lambda)|^{-2}  \,d\lambda \,dk)$, the following Fourier inversion formula holds for almost every $gK \in \mathbb{X}$ (\cite[Ch. III]{Helgason_GA})
\begin{align*}
f(gK)= \frac{1}{|W|} \int_{\fa^*} \int_K \widetilde{f}(\lambda, k)~e^{-(i\lambda+\rho)H(g^{-1}k)} ~ |c(\lambda)|^{-2}\,d\lambda~dk,
\end{align*}
 where $c(\lambda)$ is the Harish-Chandra's $c$-function. Moreover, the Plancherel theorem states \cite[Ch. III, Theorem 1.5]{Helgason_GA}
\begin{align}\label{Planc}
    \int_{\X} f_1(x) \overline{f_2(x)} \,dx= \frac{1}{|W|} \int_{\fa^*} \int_K \widetilde{f_1}(\lambda,k) \overline{\widetilde{f_2}(\lambda,k)}  |c(\lambda)|^{-2} \, dk \, d\lambda,
\end{align}
for all $f_1, f_2 \in C_c^{\infty}(\X)$. We also require the following estimates \cite[(2.2)]{CGM93} (see also \cite[Ch. IV, prop 7.2]{He_00})
\begin{equation}\label{est_c-2,hr}
\begin{aligned}
    |c(\lambda)|^{-2} \leq C |\lambda|^{\nu -l} (1+ |\lambda|)^{n-\nu}, \quad \text{for all } \lambda \in \mathfrak{a}^*.
\end{aligned}
\end{equation}
If $f$ is $K$-biinvariant, meaning  $f(k_1gk_2)=f(g)$ for all  $k_1, k_2\in K$ and $g \in G$, then  $\widetilde{f}$ does not depend on $K$ and the formula \eqref{defn:hft} simplifies to
\begin{equation*}
\widetilde{f}(\lambda, k)=\widehat{f}(\lambda):=\int_G f(g) \phi_{-\lambda}(g) \,d g,
\end{equation*}
for all $\lambda \in \mathfrak{a}^*, k \in K$, where
\begin{align*}
\phi_\lambda(g)=\int_K e^{-(i \lambda+\rho)H(g^{-1} k)} d k, \quad \lambda \in \mathfrak{a}_{\mathbb{C}}^*,
\end{align*} is the Harish-Chandra's elementary spherical function \cite{HC58}.

We now focus on the rank one case and recall some results from \cite{RS_09} in the context of rank one symmetric spaces. In this scenario, it is well known that $\Sigma$ is either of the form $\{ -\alpha , \alpha\}$ or $\{-2\alpha,- \alpha , \alpha, 2\alpha \}$. To simplify notation, we choose $H_0$ as the unique element in $\mathfrak{a}$ such that $\alpha(H_0) = 1$. This allows us to identify the complex plane $\mathbb{C}$ with $\mathfrak{a}_{\mathbb{C}}^*$ using the map $\lambda \mapsto \lambda \cdot \alpha$. This map also identifies $\R$ with $\mathfrak{a}^*$ (see \cite{Ion02} for more details). In their work, Ray and Sarkar \cite{RS_09} established the following Hausdorff-Young inequality in the context of harmonic $NA$ groups for dimensions greater or equal to $3$. By following similar calculations, we note that one can prove the same result for rank one symmetric spaces of noncompact type with $n\geq 3$. Moreover, the result is also valid for the case when $n=2$.
\begin{theorem}[{{\cite[Theorem 4.6]{RS_09}}}]\label{thm_RS_HYinq}
  Let $1\leq p\leq 2$.  Then for $p\leq q\leq p'$, we have
    \begin{align}\label{HY_ineq}
        \left( \int_{\R} \left( \int_K \left| \widetilde {f}(\lambda +i\rho_q, k)  \right|^q \, dk \right)^{\frac{p'}{q}} |c(\lambda)|^{-2} d\lambda \right) ^{\frac{1}{p'}} \leq C_{p,q} \|f\|_{L^p(\X)},
    \end{align}
    where $ \rho_p= (2/p-1)\rho.$
    The case  $p=q=2$ is a weakening of the Plancherel theorem. We also note that
when $q=p' $, the result best resembles the classical Hausdorff-Young inequality at
the lower boundary of the strip $S_p:= \{\lambda \in \C : |\Im \lambda| \leq \rho_p \}$.
\end{theorem}

\begin{theorem}[{{\cite[Theorem 4.2]{RS_09}}}]\label{RS_4.2}
    Let $1\leq p<2$. Then for $p<q<p'$ and for $\lambda \in \C$ with $\Im  \lambda  = \rho_q$, we have 
\begin{align*}
    \left( \int_{K} |\widetilde{f}(\lambda, k)|^q dk\right)^{\frac{1}{q}} \leq C_{p,q} \|f\|_{L^p(\X)},
\end{align*}
for all $f \in L^p(\X)$. Moreover, when $p=1$, then $q \in [1,\infty] $ and $C_{p,q}=1$.
\end{theorem}
We also require the sharp estimate  \cite[Lemma 4.8]{RS_09} of $|c(\lambda)|^{-2}$, which can be obtained from \cite[Lemma 4.2]{ST78} the explicit expression  of 
 $|c(\lambda)|^{-2}$.  
\begin{lemma}\label{est_c-2}
   Let $\X$ be a rank one symmetric space of noncompact type. Then we have 
    \begin{align}\label{sharp_c-2}
        |c(\lambda)|^{-2} \asymp \lambda^2 (1+|\lambda|)^{n-3}, \quad \text{for all } \lambda \in \R.
    \end{align}
\end{lemma}

   \subsection{Mixed norm spaces}\label{subsec: mix_norm}  Let $\left(X_j, dx_j\right), j=0,1$, be two $\sigma$-finite measure spaces and $(X, dx)$ their product measure space. For an ordered pair $P=\left(p_0, p_1\right) \in[1, \infty] \times[1, \infty]$ and a measurable function $f\left(x_0, x_1\right)$ on $(X, \mu)$,  we define the mixed norm $\left(p_0, p_1 \right)$ of $f$ as
\begin{align}\label{defn_mix_norm}
\|f\|_P=\|f\|_{{\left(p_0, p_1\right)}}=\left(\int_{X_0}\left(\int_{X_1}\left|f\left(x_0, x_1\right)\right|^{p_0} d x_1\right)^{p_1 / p_0} d x_0\right)^{1 / p_1} .
\end{align}
 For two such ordered pairs $P=\left(p_0, p_1\right)$ and $Q=\left(q_0, q_1\right)$, we write $1 / R=(1-\theta) / P+\theta / Q$, $0<\theta<1$ to mean $1 / r_0=(1-\theta) / p_0+\theta / q_0$ and $1 / r_1=(1-\theta) / p_1+\theta / q_1$, where $R=\left(r_0, r_1\right)$. We have the following analytic interpolation for the mixed norm spaces.

Let us consider two product measure spaces $X=X_0 \times X_1$ and $Y=Y_0 \times Y_1$. Let $d x$ and $d y$ denote respectively the (product) measures on $X$ and $Y$. Let $T_z $ be an analytic family of linear operators between $X$ and $Y$ of admissible growth, defined in the strip $\{ z \in \C : 0 \leq  \Re z \leq 1 \}$. We suppose that for all finite linear combinations of characteristic functions of rectangles of finite measures $f$ on $X$:
\begin{align*}
\left\|T_{j+ i \xi}(f)\right\|_{Q_j} \leq A_j(\xi)\|f\|_{P_j} 
\end{align*}
for $P_j, Q_j \in[1, \infty] \times[1, \infty]$ such that $\log \left|A_j(\xi)\right| \leq A e^{a|\xi|}, a<\pi, j=0, 1$. Let $1 / R=(1-\theta) / P_0+\theta / P_1$ and $1 / S=(1-\theta) / Q_0+\theta / Q_1, 0<\theta<1$. Then it follows similarly as in \cite[Page 313, Theorem 1]{MR0126155} (see also \cite[Theorem 1]{Stein_56}) that
\begin{align}\label{eqn_ana_int}
\left\|T_\theta(f)\right\|_S \leq A_{\theta}\|f\|_R.
\end{align}
We refer the reader to \cite{MR0126155} for more details about the mixed norm spaces. However, we would like to mention that the order in which the mixed norms are taken in this article differs from that in \cite{MR0126155}. 

\section{Hausdorff-Young-Paley inequality on the unitary dual}\label{sec_HYP_UD}

Classical results such as the Hausdorff-Young inequality have long illustrated the relationship between the norms of a function and its Fourier transform, establishing a foundational framework for more advanced inequalities. These inequalities have been extensively studied and well-understood in Euclidean contexts, and have even been extended to non-Euclidean settings (see \cite{CMMP19}). However, extending these results to symmetric spaces of noncompact type presents unique challenges and opportunities.

In Euclidean settings, the Hausdorff-Young inequality is proved using the Plancherel theorem and the strong type $(1,\infty)$ of the classical Fourier transform operator, which follows from the boundedness of the kernel $e^{-i \langle \xi, x \rangle}$ used to define the classical Fourier transform. A similar argument can be applied in the $K$-biinvariant setting, as seen in \cite{EK82, EKT87} and \cite[Theorem 2.1]{CGM93}, where the authors utilized the fact that the spherical function is bounded on the Helgason-Johnson strip, particularly on the unitary dual $\mathfrak{a}^*$. However, this analogy breaks down in symmetric spaces because the (Poisson) kernel used to define the Fourier transform $\widetilde{f}$ on $\X$ is not bounded. This necessitates a different approach, one that utilizes restriction inequalities. Ray and Sarkar \cite{RS_09} provided a mixed norm version of the Hausdorff-Young inequality for harmonic $NA$ groups, particularly in rank one symmetric spaces of noncompact type. We now prove the Hausdorff-Young inequality and its dual version (Theorem \ref{thm_HYP_uni}) for general symmetric spaces of noncompact type.

\begin{proof}[\textbf{Proof of Theorem \ref{thm_HY_UD}}]
   We will begin by considering the following measure spaces $(\X,dx)$ and
   $(\mathfrak{a}^* \times K, $ $ |W|^{-1}|c(\lambda)|^{-2}\,d\lambda \, dk )$.  Let us write the Fourier transform as a linear operator $\mathcal{H}$ on $C_c^\infty(\mathbb X)$ by 
   \begin{align*}
       \mathcal{H}(f)(\lambda,k) =  \widetilde{f}(\lambda,k)
   \end{align*}
   for all $\lambda \in \fa^*$, $k\in K$.
    Then for all $\lambda \in \mathfrak{a}^*$, we can write using Minkowski's integral inequality
    \begin{align*}
        \left( \int_K  \left| \widetilde{f}(\lambda,k)\right|^2 \,dk \right)^{\frac{1}{2}}  & = \left(\int_K  \left| \int_G f(x) e^{(i\lambda-\rho) H(x^{-1}k)} \,dx \right|^2 \,dk \right)^{\frac{1}{2}}\\
        & \leq \int_G |f(x)|  \left(\int_K e^{-2\rho H(x^{-1} k)}\, dk\right)^{\frac{1}{2}} dx.
    \end{align*}
    Using \eqref{F_C(K)=1} in the inequality above, we have the following analogue of the restriction estimate 
    \begin{align}\label{f_L2K<L1}
        \left( \int_K  \left| \widetilde{f}(\lambda,k)\right|^2 \,dk \right)^{\frac{1}{2}}  & \leq  \|f\|_{L^1(\X)}
    \end{align}
    for all $\lambda \in \fa^*$. So we can write
    \begin{align}\label{T_1,inf}
        \|\mathcal{H}(f)\|_{(2,\infty)} \leq  \|f\|_{L^1(\X)},
    \end{align}
    where $\|\mathcal{H}(f)\|_{(2,\infty)}$ denotes the mixed norm of $\mathcal{H}(f)$ defined in \eqref{defn_mix_norm}.    Next, by the Plancherel theorem \eqref{Planc}, we have
    \begin{align}\label{T_2,2}
        \| \mathcal{H}(f)\|_{(2,2)} =  \left(\frac{1}{|W|}\int_{\fa^*} \int_K |\widetilde{f}(\lambda,k)|^2 |c(\lambda)|^{-2} \, dk \, d\lambda \right)^{\frac{1}{2}}=  \|f\|_{L^2(\X)}.
    \end{align}
From \eqref{T_1,inf} and \eqref{T_2,2},  and by interpolation of mixed norm spaces, we derive \eqref{HY_ud_int}.

To prove \eqref{eqn_HYP_dual-int}, we will use the duality argument. Let $p \in [2,\infty)$, then for $f\in C_c^{\infty}(\X)$, we can write from \eqref{Planc} 
\begin{align*}
    \|f\|_{L^p(\X)}&=\sup_{\|h\| _{L^{p'}(\X)} \leq 1} \left| \int_{\X} f(x) \overline{h(x)} dx\right|\\ &=  \sup_{\|h\| _{L^{p'}(\X)}\leq 1}  \left|\frac{1}{|W|} \int_{\fa^*} \int_K \widetilde{f}(\lambda,k) \overline{\widetilde{h}(\lambda,k)}  |c(\lambda)|^{-2} \, dk \, d\lambda \right|.
\end{align*}
By successive applications of  H{\"o}lder's inequality, we obtain from above
\begin{align*}
     \|f\|_{L^p(\X)}& \leq \sup_{\|h\| _{L^{p'}(\X)} \leq 1}  \left(  \frac{1}{|W|}\int_{\mathbb{\fa^*}} \left(\int_K|\widetilde{f}(\lambda, k)|^{2} \, dk\right)^{\frac{p'}{2}} |c(\lambda)|^{-2} \,d \lambda \right)^{\frac{1}{p'}} \\
     & \hspace{6cm} \cdot\left( \frac{1}{|W|}\int_{\mathbb{\fa^*}} \left(\int_K|\widetilde{h}(\lambda, k)|^{2} \, dk\right)^{\frac{p}{2}} |c(\lambda)|^{-2} \,d \lambda \right)^{\frac{1}{p}} .
\end{align*}
 We observe that since $ p'\leq 2$, we can utilize \eqref{HY_ud_int} in the inequality above to obtain
 \begin{align}
     \|f\|_{L^p(\X)}& \leq  \left(\sup_{\|h\| _{L^{p'}(\X)} \leq 1}  \|h\|_{L^{p'}(\X)} \right) \left( \frac{1}{|W|} \int_{\mathbb{\fa^*}} \left(\int_K|\widetilde{f}(\lambda, k)|^{2} \, dk\right)^{\frac{p'}{2}} |c(\lambda)|^{-2} \,d \lambda \right)^{\frac{1}{p'}}, 
 \end{align}
 which establishes \eqref{eqn_HYP_dual-int}, concluding our theorem.
\end{proof}

In the theorem above, we demonstrate that the Fourier transform on symmetric spaces is a bounded operator from $L^p(\X)$ to $L^{(2,p')}(\mathfrak{a}^* \times K, |c(\lambda)|^{-2} d\lambda\, dk)$ for $1\leq p\leq 2$. Next, we consider Paley's inequality (Theorem \ref{thm_paley_uni}), which addresses the boundedness of the Fourier transform operator from $L^p(\X)$ to $L^{(2,p)}(\mathfrak{a}^* \times K, d\mu\, dk)$, where $\mu$ is an appropriate measure on $\mathfrak{a}^*$.

\begin{proof}[\textbf{Proof of Theorem \ref{thm_paley_uni}}] For a given positive function $u$ on $\fa^*$, we consider the measure spaces $(\X, dx)$ and $(\fa^* , d\mu)$, where \begin{align}\label{def_mu} 
d\mu(\lambda) : = u(\lambda)^2 |c(\lambda)|^{-2} d\lambda.
\end{align}
We then define the following sublinear operator mapping $C_c^{\infty}(\X)$ to measurable functions on $\fa^*$ by 
\begin{align}\label{def_T_pal}
  \mathcal{T}f(\lambda) = \frac{1}{u(\lambda)}\|\widetilde{f}(\lambda,\cdot)\|_{L^2(K)}.
\end{align}
By taking the $L^2$ norm with respect to the measure $ d\mu(\lambda)$ in \eqref{def_mu}, it follows that
\begin{align*} 
   \|  \mathcal{T}f(\lambda)\|_{2}^2&= \int_{\fa^*} \frac{1}{u(\lambda)^2}\int_K|\widetilde{f}(\lambda, k)|^{2} \, dk\,  u(\lambda)^2  |c(\lambda)|^{-2} \, d\lambda\\
   &=   \int_{\fa^*} \int_K |\widetilde{f}(\lambda,k)|^2 |c(\lambda)|^{-2} \, dk \, d\lambda\\
   &= |W| \|f\|_{L^2(\X)}^2,
\end{align*}
where the last step follows from the Plancherel theorem \eqref{Planc}. Thus, we have shown that the operator $\mathcal{T}$ is bounded from $L^2(\X)$ to $L^2(\fa^*, d\mu(\lambda)) $. Our next objective is to establish that $\mathcal{T}$ is also weak-type $(1,1)$. More precisely, we will prove the following
\begin{align}\label{eqn_w1,1}
   \sup_{\alpha>0} \alpha\, \mu\left\{ \lambda \in \fa^* : \frac{1}{u(\lambda)}\|\widetilde{f}(\lambda,\cdot)\|_{L^2(K)}   >\alpha  \right\}\leq C \|u\|_{c,\infty} {\|f\|_{L^1(\X)}}.
\end{align}
Using the restriction estimate \eqref{f_L2K<L1}, we have
\begin{align*}
    \mu\left\{ \lambda \in \fa^* : \frac{1}{u(\lambda)}\|\widetilde{f}(\lambda,\cdot)\|_{L^2(K)}   >\alpha  \right\} \subset \mu\left\{ \lambda \in \fa^* : \frac{1}{u(\lambda)} \|f\|_{L^1(\X)}   >\alpha  \right\},
\end{align*}
for all $\alpha>0$. This implies that \eqref{eqn_w1,1} follows from the subsequent inequality
\begin{align}\label{sup_f_L1}
    \sup_{\alpha>0} \alpha\,\mu\left\{ \lambda \in \fa^* : \frac{1}{u(\lambda)} \|f\|_{L^1(\X)}   >\alpha  \right\} \leq C \|u\|_{c,\infty} {\|f\|_{L^1(\X)}}.
\end{align}
By taking $\gamma = \alpha/\|f\|_{L^1(\X)}$, we have from the definition of $\mu$ in  \eqref{def_mu} that
\begin{align}\label{dis_func}
    \mu\left\{ \lambda \in \fa^* : \frac{1}{u(\lambda)}   >\gamma  \right\} =   \mu\left\{ \lambda \in \fa^* : {u(\lambda)}   <\frac{1}{\gamma}  \right\}= \int\limits_{\{\lambda\in \fa^*:  u(\lambda)<\frac{1}{\gamma}\}} u(\lambda)^2 |c(\lambda)|^{-2} \,d\lambda.
\end{align}
We can rewrite the integral above as follows 
\begin{align*}
    \int\limits_{\{\lambda\in \fa^*:  u(\lambda)<\frac{1}{\gamma}\}} u(\lambda)^2 |c(\lambda)|^{-2} \,d\lambda=     \int\limits_{\{\lambda\in \fa^*:  u(\lambda)<\frac{1}{\gamma}\}} \left( \int_0^{u(\lambda)^2} d\zeta\right) |c(\lambda)|^{-2} \,d\lambda.
\end{align*}
By interchanging the order of integration, we obtain
\begin{align*}
     \int\limits_{\{\lambda\in \fa^*:   u(\lambda)<\frac{1}{\gamma}\}} \left( \int_0^{u(\lambda)^2} d\zeta\right) |c(\lambda)|^{-2} \,d\lambda=   \int_0^{\frac{1}{\gamma^2}}  \left(\int\limits_{\{\lambda\in \fa^*:  \zeta^{\frac{1}{2}}<u(\lambda)<\frac{1}{\gamma}\}} |c(\lambda)|^{-2} \,d\lambda \right)   d\zeta .
\end{align*}
By the change of variable $\zeta\rightarrow \zeta^{{2}}$ in the integral of the right hand side above yields the following 
\begin{align*}
    \int_0^{\frac{1}{\gamma^2}}  \left(\int\limits_{\{\lambda\in \fa^*:  \zeta^{\frac{1}{2}}<u(\lambda)<\frac{1}{\gamma}\}} |c(\lambda)|^{-2} \,d\lambda \right)   d\zeta &=  2     \int_0^{\frac{1}{\gamma}} \zeta \left(\int\limits_{\{\lambda\in \fa^*:  \zeta<u(\lambda)<\frac{1}{\gamma}\}} |c(\lambda)|^{-2} \,d\lambda \right)   d\zeta \\
    & \leq 2     \int_0^{\frac{1}{\gamma}}  \zeta \left(\int\limits_{\{\lambda\in \fa^*:  u(\lambda)>\zeta\}} |c(\lambda)|^{-2} \,d\lambda \right)   d\zeta.
\end{align*}
Taking the supremum in the inequality above and using \eqref{dis_func}, we derive
\begin{align*}
    \mu\left\{ \lambda \in \fa^* : \frac{1}{u(\lambda)}   >\gamma  \right\}\leq  2  \int_0^{\frac{1}{\gamma}}    \sup_{\zeta>0} \zeta \left(\int\limits_{\{\lambda\in \fa^*:  u(\lambda)>\zeta\}} |c(\lambda)|^{-2} \,d\lambda \right)  d\zeta= \frac{2\|u\|_{c,\infty}}{\gamma}.
\end{align*}
Substituting the value of $\gamma = \alpha/\|f\|_{L^1(\X)}$ in the inequality above, we obtain \eqref{sup_f_L1}, consequently establishing our claim \eqref{eqn_w1,1}. Now utilizing the Marcinkiewicz interpolation theorem, we obtain \eqref{eqn_paley}
\begin{align*}
    \left(\int_{\fa^*} |\mathcal{T}f(\lambda)|^p \, u(\lambda)^2 |c(\lambda)|^{-2} \, d\lambda \right)^{\frac{1}{p}}\leq C_p  \|u\|_{c,\infty}^{\frac{2}{p}-1} \|f\|_{L^p(\X)},
\end{align*}
for $p\in (1,2]$, concluding the proof of our theorem.
\end{proof}
Next, we use the preceding theorems to establish a Hausdorff-Young-Paley inequality by interpolating between the Hausdorff-Young inequality and the Paley inequality. This inequality will be crucial for our subsequent analysis of $L^p \rightarrow L^q$ Fourier multipliers.
\begin{theorem}\label{thm_HYP_uni}
    Let $1<p\leq 2$ and $1<p\leq b \leq p'<\infty $. Assume that $ u: \fa^* \rightarrow (0,\infty)$ is a positive function such that  \eqref{cond_u_HYP} holds.
  Then for all $f \in L^p(\X)$, we have
  \begin{align*}
      \left( \int_{\mathbb{\fa^*}} \left(\int_K|\widetilde{f}(\lambda, k)|^{2} \, dk\right)^{\frac{b}{2}} u(\lambda)^{1-\frac{b}{p'}} |c(\lambda)|^{-2} \,d \lambda \right)^{\frac{1}{b}} \leq C_p  \|u\|_{c,\infty}^{\frac{1}{b}-\frac{1}{p'}}\|f\|_{L^p(\X)}.
  \end{align*}
\end{theorem}
\begin{proof}
Suppose $1\leq p\leq 2$, then we have from Theorem \ref{thm_HYP_uni} that the linear operator defined by
\begin{align*}
    \mathcal{H}(f)(\lambda,k)=\widetilde{f}(\lambda,k)
\end{align*}
is bounded from $L^{p}(\X)$ to $L^{(2,p')}(\fa^*\times K, |c(\lambda)|^{-2} d\lambda dk$). Morever, assuming the function $u$ satisfies \eqref{cond_u_HYP}, we can utilize Theorem \ref{thm_HYP_uni} and deduce that $\mathcal{H}$ is also bounded from $L^p(\X)$ to  $L^{(2,p)}(\fa^*\times K, u(\lambda)|c(\lambda)|^{-2} d\lambda dk$). Now, we will apply the Stein-Weiss interpolation theorem for mixed norm spaces. More precisely, by employing {(real version of)} Corollary \ref{cor_sw_ana} (see also \cite[page 120]{BL76}) for the operator $\mathcal{H}$ with $Q_0=(2, p)$, $Q_1=(2,p')$, and
      \begin{align*}
      {w}_0(\lambda)= u(\lambda)^{2-p}  |c(\lambda)|^{-2}, \qquad w_1(\lambda)= |c(\lambda)|^{-2},
    \end{align*} 
    we deduce that
    \begin{align}\label{eqn_intp_T_dual}
    \left( \int_{\fa^*} \left( \int_K  \left|\widetilde {f}(\lambda, k) \right|^{2}   \,  dk \right)^{\frac{b}{2}}   w_{\theta}(\lambda) d\lambda \right) ^{\frac{1}{b}}\leq C_{b}  {\|u\|_{c,\infty}^{({\frac{2}{p}-1})(1-\theta)}} \|f\|_{L^p(\X)},
    \end{align}
    where 
    \begin{align*}
        \frac{1}{b}= \frac{1-\theta}{p}+\frac{\theta}{p'}, \qquad \text{and} \qquad  w_{\theta}(\lambda)=  {w}_0(\lambda)^{\frac{b(1-\theta)}{p}} {w}_1(\lambda)^{\frac{b\theta}{p'}}.
    \end{align*}
    Solving the expression above in $\theta\in (0,1)$, we get $\theta= \frac{b-p}{b(2-p)}$, which implies
    \begin{align*}
          \frac{b(1-\theta)}{p}= \frac{b-bp+p}{p(2-p)}= \frac{p'-b}{p'(2-p)}.
    \end{align*}
   Substituting the  expression above in $w_{\theta}$, we obtain
    \begin{align*}
        w_{\theta}(\lambda)= u(\lambda)^{1-\frac{b}{p'}}  |c(\lambda)|^{-2}.
    \end{align*}
    Consequently, we have from \eqref{eqn_intp_T_dual} that
    \begin{align*}
         \left( \int_{\fa^*} \left( \int_K  \left|\widetilde {f}(\lambda, k) \right|^{2}   \,  dk \right)^{\frac{b}{2}}    u(\lambda)^{1-\frac{b}{p'}}  |c(\lambda)|^{-2}\, d\lambda \right) ^{\frac{1}{b}}\leq C_{b}  {\|u\|_{c,\infty}^{\frac{1}{b}-\frac{1}{p'}}} \|f\|_{L^p(\X)}.
    \end{align*}
    This completes the proof of our theorem.
\end{proof}

\section{$L^p \rightarrow L^q$ Fourier  multipliers on Riemannian symmetric spaces}\label{sec_Lp-Lq_multiplier}
In this section, we present our main results on the $L^p \rightarrow L^q$ boundedness of multiplier operators on symmetric spaces of noncompact type. Before delving into these results, let us recall that using the Plancherel theorem, a multiplier operator $T_m$ for a multiplier $m$ can be defined by
\begin{align}\label{defn_mult_op_Tm}
\widetilde{T_m f}(\lambda,k) = m(\lambda) \widetilde{f}(\lambda,k) 
\end{align}
for all $\lambda \in \fa^*$, $k\in K$.
First, we prove the $L^p \rightarrow L^q$ boundedness of the Fourier multiplier operator $T_m$ (Theorem \ref{thm_Lp-Lq_int}). Then, as applications, we will establish a spectral multiplier theorem, obtain the $L^p \rightarrow L^q$ boundedness of functions of the Laplace-Beltrami operator, and present additional results in the context of symmetric spaces of noncompact type.
\begin{proof}[\textbf{Proof of Theorem \ref{thm_Lp-Lq_int}}]
    We begin with the observation that it suffices to prove the result for $p\leq q'$, as the result for the case $p\geq q'$ will follow by duality. Indeed, if $p\geq q' $, we can reduce the $L^p\rightarrow L^q$ boundedness of $T_m$ to the case other case by using the following   $$\| T_m\|_{L^p\rightarrow L^q} = \| T_m^* \|_{L^{p'}\rightarrow L^{q'}} \qquad  \text{ and } \qquad T_m^* =T_{\overline{m}}.$$
So, we assume $p\leq q'$. As $q\geq 2$,  applying \eqref{eqn_HYP_dual-int} we obtain
\begin{align*}
    \|T_m f\|_{L^q(\X)}\leq C_q  \left( \int_{\mathbb{\fa^*}} \left(\int_K|\widetilde{T_m f} (\lambda, k)|^{2} \, dk\right)^{\frac{q'}{2}} |c(\lambda)|^{-2} \,d \lambda \right)^{\frac{1}{q'}} .
\end{align*}
Now, employing Theorem \ref{thm_HYP_uni} with $b=q'$, $u=|m|^{r}$, where $1/r =1/p-1/q$, it follows that
\begin{equation}\label{Tmf_q<p} 
\begin{aligned}
    \|T_m f\|_{L^q(\X)}& \leq C_q \left( \int_{\mathbb{\fa^*}} \left(\int_K|\widetilde{ f} (\lambda, k)|^{2} \, dk\right)^{\frac{q'}{2}} |m(\lambda)|^{q'} |c(\lambda)|^{-2} \,d \lambda \right)^{\frac{1}{q'}}\\
    &\leq C_{p,q} \left(\sup\limits_{\alpha>0}\,\alpha \int\limits_{\{ \lambda \in \fa^* : |m(\lambda)|^r>\alpha \}   }  |c(\lambda)|^{-2} \, d\lambda\right)^{\frac{1}{r}} \|f\|_{L^p(\X)},
\end{aligned}
\end{equation}
where in the last inequality, we utilized the fact $1/q'-1/p'= 1/p-1/q$. Furthermore, by writing
    \begin{align*}
        \left(\sup\limits_{\alpha>0}\,\alpha \int\limits_{\{ \lambda \in \fa^* : |m(\lambda)|^r>\alpha \}   }  |c(\lambda)|^{-2} d\lambda\right)^{\frac{1}{r}} &= \left(\sup\limits_{\alpha>0}\,\alpha \int\limits_{\{ \lambda \in \fa^* : |m(\lambda)|>\alpha^{\frac{1}{r}} \}   }  |c(\lambda)|^{-2} d\lambda\right)^{\frac{1}{r}} \\
        & = \left(\sup\limits_{\alpha>0}\,\alpha^r \int\limits_{\{ \lambda \in \fa^* : |m(\lambda)|>\alpha \}   }  |c(\lambda)|^{-2} d\lambda\right)^{\frac{1}{r}}\\
        &= \sup\limits_{\alpha>0}\,\alpha \left( \int\limits_{\{ \lambda \in \fa^* : |m(\lambda)|>\alpha \}   }  |c(\lambda)|^{-2} d\lambda\right)^{\frac{1}{r}},
    \end{align*}
    we obtain from \eqref{Tmf_q<p}  that
    \begin{align*}
         \|T_m f\|_{L^q(\X)}& \leq C_{p,q} \left( \sup\limits_{\alpha>0}\,\alpha \left( \int\limits_{\{ \lambda \in \fa^* : |m(\lambda)|>\alpha \}   }  |c(\lambda)|^{-2} d\lambda\right)^{\frac{1}{p}-\frac{1}{q}} \right) \|f\|_{L^p(\X)}
    \end{align*}
    for all $f \in L^p(\X)$.
\end{proof}
We will now utilize the connection between the $L^p \rightarrow L^q$ norm of multiplier operators, for $1 < p \leq 2 \leq q < \infty$, and Harish-Chandra's $c$-function to determine the norm estimate of the spectral multiplier operator. Specifically, by using the sharp estimate of $|c(\lambda)|^{-2}$, we will establish the norm estimate of the spectral multiplier operator and prove Theorem \ref{thm_spec_mul_int}.
\begin{proof}[\textbf{Proof of Theorem \ref{thm_spec_mul_int}}] By comparing the definition of $\varphi(-\Delta)$ with $T_m$ \eqref{defn_mult_op_Tm}, we note that $\varphi(-\Delta)$ is a Fourier multiplier operator $T_m$  with
    \begin{align}\label{eqn_m=varphi}
        m(\lambda)= \varphi \left(|\lambda|^2+|\rho|^2\right) \quad \text{for all } \lambda \in \fa^*\setminus \{0\}.
    \end{align}
    We observe that since $s\rightarrow |\varphi(s)|$ is a continuous and monotonic function on the interval  $(|\rho|^2, \infty)$, $\lim_ {s\rightarrow{|\rho|^2}^+} |\varphi(s)|$ always exists and can be finite or infinite. Therefore, by applying Theorem \ref{thm_Lp-Lq_int} and noting that the value of an integral is not affected by a set of measure zero, we derive
    \begin{align}\label{est_varom}
        \|\varphi(-\Delta)\|_{L^p(\X) \rightarrow L^q(\X)} \lesssim  \sup_{\alpha>0}  \alpha \left( \int\limits_{\{  \lambda \in \fa^*\setminus \{0\} : |\varphi\left({|\lambda|^2+|\rho|^2}\right)|> \alpha\}}|c(\lambda)|^{-2} d\lambda \right)^{\frac{1}{p}-\frac{1}{q}}. 
    \end{align}
    By our hypothesis $|\varphi|$ is continuously decreasing over the interval $[|\rho|^2, \infty)$. Consequently, the inner integral in the inequality  \eqref{est_varom} becomes zero when $\alpha \geq |\varphi(|\rho|^2)| $. Therefore, we can say the right-hand side of the inequality above \eqref{est_varom} is the same as 
    \begin{align*}
          \sup_{0<\alpha< |\varphi(|\rho|^2)|}  \alpha \left( \int\limits_{\{  \lambda \in \fa^* \setminus \{0\} : |\varphi\left({|\lambda|^2+|\rho|^2}\right)|> \alpha\}}|c(\lambda)|^{-2} d\lambda 
        \right)^{\frac{1}{p}-\frac{1}{q}}.
    \end{align*}   
   Now, choose any $\alpha \in (0, |\varphi(|\rho|^2)|)$. By the hypothesis on $\varphi$, there exists some $s \in (|\rho|^2, \infty)$ such that $\alpha = |\varphi(s)|$. Utilizing this and the monotonicity of $|\varphi|$, we can write
    \begin{equation}\label{est_of_spec_opn}
        \begin{aligned}
        & \sup_{\alpha>0}  \alpha \left( \int\limits_{{\{  \lambda \in \fa^* \setminus \{0\}  : |\varphi\left({|\lambda|^2+|\rho|^2}\right)|> \alpha\}}} |c(\lambda)|^{-2} d\lambda \right)^{\frac{1}{p}-\frac{1}{q}} \\
        &= \sup_{|\varphi(s)|< |\varphi(|\rho|^2)|}  |\varphi(s)| \left( \int\limits_{\{  \lambda \in \fa^* \setminus \{0\}  : |\varphi\left({|\lambda|^2+|\rho|^2}\right)|>| \varphi(s)|\}}  |c(\lambda)|^{-2} d\lambda \right)^{\frac{1}{p}-\frac{1}{q}}\\
         &= \sup_{s> |\rho|^2}  |\varphi(s)| \left( \int\limits_{\{  \lambda \in \fa^* \setminus \{0\}  :|\lambda| <   \sqrt{s-|\rho|^2}\}} |c(\lambda)|^{-2} d\lambda \right)^{\frac{1}{p}-\frac{1}{q}}.
    \end{aligned}    
    \end{equation}
Using the estimate \eqref{est_c-2,hr} of $|c(\lambda)|^{-2}$ and considering cases where $|\lambda|$ is near and away from zero, we can write from \eqref{est_varom} and the inequality above
    \begin{align*}
        &\|\varphi(-\Delta)\|_{L^p(\X) \rightarrow L^q(\X)}   \\
        &\leq C \max \begin{cases}  \sup\limits_{|\rho|^2 <s \leq 1+ |\rho|^2}  |\varphi(s)| 
              \left( \int_0^{ \sqrt{s-|\rho|^2}}   r^{\nu -l} r^{l-1}  dr \right)^{\frac{1}{p}-\frac{1}{q}},\\
               \sup\limits_{s \geq  1+|\rho|^2}  |\varphi(s)|  \left( \int_0^{1} r^{\nu -l} r^{l-1} dr +\int_1^{\sqrt{s-|\rho|^2}} r^{n-l} r^{l-1} \,dr \right)^{\frac{1}{p}-\frac{1}{q}}.
          \end{cases}  \\
          & \leq   C \max \begin{cases}
             \sup\limits_{ |\rho|^2 <s\leq 1+ |\rho|^2}  |\varphi(s)| \left( s-  |\rho|^2 \right)^{ \frac{\nu}{2}(\frac{1}{p}-\frac{1}{q})},\\
           \sup\limits_{s\geq  1+|\rho|^2}  |\varphi(s)|   \left( s- |\rho|^2\right)^{\frac{n}{2}(\frac{1}{p}-\frac{1}{q})}.
          \end{cases}
    \end{align*}
    This concludes the proof of our theorem.
\end{proof}
\begin{remark}\label{rem_m<est}  
\begin{enumerate}
    \item We observe that if $m_o,m: \mathfrak{a}^* \rightarrow \C$ are two measurable functions such that $|m_o| \leq  |m|$, then the multiplier operator $T_{m_o}$ satisfies the following estimate for $1<p\leq 2\leq q<\infty$
   \begin{align}\label{eqn_T_mo_Tm<_Lp-Lq}
          \|T_{m_o} f\|_{L^q(\X)}& \leq C_{p,q} \left( \sup\limits_{\alpha>0}\,\alpha \left( \int\limits_{\{ \lambda \in \fa^* : |m(\lambda)|>\alpha \}   }  |c(\lambda)|^{-2} d\lambda\right)^{{\frac{1}{p}-\frac{1}{q}}} \right) \|f\|_{L^p(\X)}
     \end{align}
     for all $f\in L^p(\X)$.
     
\item \label{rem_phi<}  It follows from the previous remark that if $\varphi_o,\varphi: (|\rho|^2, \infty) \rightarrow \C$ are two measurable functions such that $|\varphi_o| \leq  |\varphi|$ and $\varphi$ satisfies all the hypotheses of Theorem \ref{thm_spec_mul_int}, then the spectral multiplier operator $\varphi_o(-\Delta)$ satisfies the following estimate for $1<p\leq 2\leq q<\infty$
         \begin{equation}\label{est_spec_op2}
      \begin{aligned}
        \|\varphi_o(-\Delta)\|_{L^p(\X) \rightarrow L^q(\X)}  & \leq C_{p,q}   \max \begin{cases}
             \sup\limits_{ |\rho|^2 <s\leq 1+ |\rho|^2}  |\varphi(s)| \left( s-  |\rho|^2 \right)^{ \frac{\nu}{2}(\frac{1}{p}-\frac{1}{q})},\\
           \sup\limits_{s\geq  1+|\rho|^2}  |\varphi(s)|   \left( s- |\rho|^2\right)^{\frac{n}{2}(\frac{1}{p}-\frac{1}{q})}.
          \end{cases}
   \end{aligned}
   \end{equation}
    \item For hyperbolic spaces $\mathbb{H}^n= \{(x,y): x\in \R^{n-1}, y>0\}$ $(n\geq 2)$, it is known that $|\rho|=(n-1)/2$ and  the pseudo-dimension $\nu$ always equals $3$. Therefore, in the specific case of hyperbolic spaces, we have demonstrated that if $\varphi$ satisfies the hypotheses of Theorem \ref{thm_spec_mul_int}, the corresponding spectral multiplier operator $\varphi(-\Delta)$ satisfies the following estimate for $1< p\leq 2\leq q<\infty$
     \begin{equation*}
      \begin{aligned}
        \|\varphi(-\Delta)\|_{L^p(\mathbb{H}^n) \rightarrow L^q(\mathbb{H}^n)}  & \leq C_{p,q} \max \begin{cases} \sup\limits_{(\frac{n-1}{2})^2 < s\leq 1+ (\frac{n-1}{2})^2}  |\varphi(s)| 
              \left( s-  (\frac{n-1}{2})^2 \right)^{ \frac{3}{2}(\frac{1}{p}-\frac{1}{q})},\\
             \sup\limits_{s\geq 1+ (\frac{n-1}{2})^2}  |\varphi(s)|  \left( s- (\frac{n-1}{2})^2\right)^{\frac{n}{2}(\frac{1}{p}-\frac{1}{q})} .
          \end{cases}
   \end{aligned}
   \end{equation*}
\end{enumerate}
\end{remark}
\subsection{Applications of multiplier theorems}\label{subs_App_Lp_Lq}
In this section, we explore the applications of $L^p \rightarrow L^q$ multipliers to various problems, highlighting their versatility in different contexts. Specifically, we will discuss the role of Theorem \ref{thm_spec_mul_int} in establishing the boundedness of functions of the Laplace-Beltrami operator, deriving Sobolev-type inequalities, and obtaining estimates for heat semigroups.
  \subsubsection{The potentials $(z I- \Delta)^{-\frac{\sigma}{2}}$}\label{subsec_potential} Here we consider the family of operators 
\begin{align}\label{def_T_z,s}
\mathcal{R}_z^{\sigma} = (z I- \Delta)^{-\frac{\sigma}{2}}.
\end{align}
for $z, \sigma \in \C$ with $\Re z \geq -|\rho|^2$ and $\Re \sigma \geq 0$. Anker \cite[Sec. 4]{An92} realized the operator $\mathcal{R}_z^{\sigma}$ by convolution on the right with the $K$-biinvariant tempered distribution $\mathcal{K}_{z,\sigma}$ on $G$, whose Fourier transform is 
\begin{align}\label{eqn_m(z,s)}
m_{z,\sigma} (\lambda)=  (|\lambda|^2+|\rho|^2+z)^{-\frac{\sigma}{2}}, \qquad \lambda \in \mathfrak{a}^{*}.    
\end{align}
Anker \cite[Theorem 4.1]{An92} obtained sharp pointwise estimates of the kernels $\mathcal{K}_{z,\sigma}$, and as a corollary \cite[Cor. 4.2]{An92}, he proved the $L^p \rightarrow L^q$ estimates of the operators $\mathcal{R}_z^{\sigma}$. Here, we establish  the $L^p \rightarrow L^q$ $(1 < p \leq 2 \leq q < \infty)$ boundedness of the operators $\mathcal{R}_z^{\sigma}$ and  prove Corollary \ref{cor_potentials} using Theorem \ref{thm_spec_mul_int}. First, we note that $m_{z,\sigma}$ can be expressed as a spectral multiplier $\varphi_{z,\sigma}$, where 
\begin{align}\label{def_vaphi_z,s}
     \varphi_{z,\sigma}(s) : =  (s+z)^{-\frac{\sigma}{2}}, \qquad s\in [|\rho|^2, \infty).
\end{align}
It follows that for $\Re z \geq -|\rho|^2$ and $\Re \sigma>0$,  the function $s \rightarrow |\varphi_{z,\sigma}(s)|$
\begin{align}\label{eqn_|phi_z,s|}
    |\varphi_{z,\sigma}(s)|= \left((s+\Re z)^2 + (\Im z)^2)\right)^{-\frac{\Re \sigma}{4}}
\end{align} 
 decreases monotonically and continuously on $[|\rho|^2, \infty)$, with $\lim_{s\rightarrow \infty} |\varphi_{z, \sigma}(s)| = 0$. In particular, $\varphi_{z,\sigma}$ satisfies all the hypotheses of Theorem \ref{thm_spec_mul_int} provided $\Re z \geq -|\rho|^2$ and $\Re \sigma>0$. Therefore, applying Theorem \ref{thm_spec_mul_int}, we have
\begin{align}\label{eqn_Tz,s_lpq}
    \|\mathcal{R}_z^{\sigma}\|_{L^p(\X) \rightarrow L^q(\X)} \leq   C_{p,q}   \max \begin{cases}
             \sup\limits_{ |\rho|^2 <s\leq 1+ |\rho|^2}  |\varphi_{z,\sigma}(s)| \left( s-  |\rho|^2 \right)^{ \frac{\nu}{2}(\frac{1}{p}-\frac{1}{q})},\\
           \sup\limits_{s \geq 1+ |\rho|^2}  |\varphi_{z,\sigma}(s)|  \left( s- |\rho|^2\right)^{\frac{n}{2}(\frac{1}{p}-\frac{1}{q})}.
          \end{cases}
\end{align}
Given that $1<p \leq 2 \leq q<\infty$,   from \eqref{eqn_|phi_z,s|}, it is clear that the supremum on the right-hand side of \eqref{eqn_Tz,s_lpq} is finite for $|\rho|^2 <  s\leq 1+|\rho|^2$ whenever $\Re z > -|\rho|^2$. Moreover, if $\Re z = -|\rho|^2$, the same will hold provided $\Im z \not =0 $.  For $s >1+ |\rho|^2$, the supremum will be finite if $\Re \sigma \geq  n \left( {1}/{p} -{1}/{q}\right).$ When $z=-|\rho|^2$, to handle the singularity of the function $ |\varphi_{z,\sigma}(s)|$ at $s=-|\rho|^2$, we additionally need to assume $\Re \sigma \leq  \nu \left( {1}/{p} -{1}/{q}\right)$.    
Thus Theorem \ref{thm_spec_mul_int} and the calculations above lead us to the proof of Corollary \ref{cor_potentials}, which provides sufficient conditions on $z$ and $\sigma$ for which the $L^p\rightarrow L^q$ boundedness will hold for $\mathcal{R}_{z}^\sigma$.

Furthermore, to find an estimate of $  \|\mathcal{R}_\zeta^{\sigma}\|_{L^p(\X) \rightarrow L^q(\X)}$ for $\zeta \in \R$ with $\zeta>-|\rho|^2$, let us define 
\begin{align*}
    \mathcal{I}_{\zeta, \sigma}(s)= \frac{\left( s-  |\rho|^2 \right)^{ \frac{\alpha}{2r}}}{\left(s+\zeta\right)^{\frac{\Re \sigma}{2}}}, \quad s>|\rho|^2,
\end{align*}
where $\alpha>0$  is fixed, $1/r =(1/p-1/q)$, and $\Re \sigma>\alpha/r$. By a simple calculation, it follows that the function $s\rightarrow \mathcal{I}_{\zeta, \sigma}(s)$ attains its maximum at the point $s_0 = (\zeta \alpha + \Re \sigma r |\rho|^2)/(\Re \sigma r-\alpha)$, and 
\begin{align}\label{sup_of_Is}
     \sup\limits_{s>|\rho|^2}\mathcal{I}_{\zeta, \sigma}(s)= \frac{ \alpha^{\frac{\alpha}{2r}} }{(\Re \sigma r)^{\frac{\Re \sigma}{2}}}  \left(\frac{\zeta+|\rho|^2}  { \Re \sigma r -\alpha} \right)^{\frac{\alpha}{2r}-\frac{\Re \sigma}{2}}.
\end{align}
Plugging \eqref{sup_of_Is} into the inequality \eqref{eqn_Tz,s_lpq}, it follows that for  $\Re \sigma>\max\{\nu,n\}/r$ there exists a constant $C=C(p,q,\sigma)>0$ such that
\begin{align}\label{op_est_R_z}
    \|\mathcal{R}_\zeta^{\sigma}\|_{L^p(\X) \rightarrow L^q(\X)} \leq   C \max\left\{(\zeta+|\rho|^2)^{\frac{\nu}{2r}-\frac{\Re \sigma}{2}}, (\zeta+|\rho|^2+1)^{-\frac{\Re \sigma}{2}},(\zeta+|\rho|^2)^{\frac{n}{2r}-\frac{\Re \sigma}{2}} \right\}.
\end{align}
In particular, for rank one symmetric spaces with dimension $n\geq 3$, we obtain the following
\begin{align}\label{est_pot_norm}
    \|\mathcal{R}_\zeta^{\sigma}\|_{L^p(\X) \rightarrow L^q(\X)} \leq   C \begin{cases}
        (\zeta+|\rho|^2)^{\frac{3}{2}(\frac{1}{p}-\frac{1}{q})-\frac{\Re \sigma}{2}}, \quad & \text{if }-|\rho|^2< \zeta \leq 1-|\rho|^2,\\
        (\zeta+|\rho|^2)^{\frac{n}{2}(\frac{1}{p}-\frac{1}{q})-\frac{\Re \sigma}{2}}, \quad &\text{if } \zeta \geq 1-|\rho|^2,
    \end{cases} 
\end{align}
provided $\Re \sigma>n (1/p-1/q)$. We note that following a similar calculation, one can derive an estimate for $ \|\mathcal{R}_{z}^{\sigma}\|_{L^p(\X) \rightarrow L^q(\X)}$, analogous to \eqref{op_est_R_z}, for $z \in \mathbb{C}$ with $\Re(z) > -|\rho|^2$.

% \begin{remark}
%     In Theorem \ref{thm_spec_mul_int}, if we allow $s \rightarrow \varphi(s)$ to have a singularity at $s =|\rho|^2$ and assume that  $\varphi: (|\rho|^2, \infty)  \rightarrow \C$ is a function such that $|\varphi|$ decreases monotonically and continuously on $(|\rho|^2, \infty)$, with $\lim_{s\rightarrow \infty} |\varphi(s)| = 0$, we can follow the same steps to show that the spectral multiplier operator $\varphi(-\Delta)$ satisfies the estimate \eqref{est_spec_op}. However, in this case, the supremum in the first line of \eqref{est_spec_op} may not always be finite. For example, when $z=-|\rho|^2$,  to address the singularity of $ |\varphi_{z,\sigma}(s)|$ at $s=|\rho|^2$ (see \eqref{eqn_|phi_z,s|}), we need to additionally assume that $\Re \sigma \leq  \nu \left( {1}/{p} -{1}/{q}\right)$ (with $\Re \sigma \geq  n \left( {1}/{p} -{1}/{q}\right)$) in order to conclude that$(-\Delta -|\rho|^2)^{-\frac{\sigma}{2}}$ is bounded from $L^p(\X)$ to $L^q(\X)$.  
% \end{remark}
Natural examples of multipliers are provided by exponential powers of modified Laplacians on $\X$. In the rank-one case, Ionescu \cite[Theorem 9]{Ion02} obtained sharp $L^p \rightarrow L^p$ estimates for the operator $L_{p,\sigma}$ when $\sigma$ is purely imaginary and $p \in (1,2) \cup (2,\infty)$, where the operator $L_{p,\sigma}$ defined  by 
\begin{align*}
    \widetilde{L_{p,\sigma} f}(\lambda, k) = -(|\lambda|^2 + |\rho_p|^2)^{-\frac{\sigma}{2}} \widetilde{f}(\lambda,k), \qquad \lambda \in \mathfrak{a}^*, k\in K,
\end{align*}
for any smooth compactly supported function $f $ on $\X$. In the following, we prove the $L^p\rightarrow L^q$ boundedness of the operator $L_{p,\sigma}$ within the context of general noncompact type symmetric spaces.
\begin{corollary}
     Let $1<p\leq 2 \leq q< \infty$ and $ q_0   \in [1,2) \cup (2,\infty]$. If $\Re \sigma \geq n (1/p-1/q)$, then the operator $L_{q_0,\sigma}$ is bounded from $L^p(\X)$ to $L^q(\X).$ 
\end{corollary}
    The proof follows from Corollary \ref{cor_potentials} by  observing the following   
    \begin{align*}
        L_{q_0,\sigma}= \left(-(|\rho|^2 - |\rho_{q_0}|^2)- \Delta \right)^{-\frac{\sigma}{2}}.
    \end{align*}

\begin{remark}\begin{enumerate}
\item The proof of Corollary \ref{cor_sobolev} follows similarly as in \eqref{eqn_Tz,s_lpq}.
 \item More generally, the following result also holds: Let $1<p\leq 2 \leq q< \infty$, $ q_0   \in [1,2) \cup (2,\infty]$, and $\sigma , \kappa \in \R$. If $\sigma-\kappa \geq n (1/p-1/q)$, then there exists a constant $C>0$ such that  for all $ f\in C_c^{\infty}(\X)$ 
      \begin{align*}
        \left\|  L_{q_0,\kappa}  f\right\|_{L^q(\X)} \leq C   \left\|  L_{q_0,\sigma} f\right\|_{L^p(\X)}.
    \end{align*}
\end{enumerate}
\end{remark}
\subsubsection{Estimates for heat semigroups}
 We now consider the following $\sigma$-fractional heat equation for $\sigma>0$ and $ 1<p\leq 2:$ 
\begin{equation}
\begin{aligned}\label{23vis}
     \frac{\partial}{\partial t}u(t,x) +(-\Delta)^{\sigma}  u(t,x)&=0, \qquad t >0, \\
        u(0,x)&=u_0(x) \in L^p(\X). 
\end{aligned}
\end{equation}
It is known that for $t>0,$ $$u(t, x)=e^{-t (-\Delta)^{\sigma}}u_0(x)$$ is a solution to the initial value problem \eqref{23vis}. Consequently, it represents the spectral multiplier operator  \eqref{def_spec_m} with the symbol
 \begin{align*}
     \hspace{2cm} \varphi\left(s \right)= e^{- t s^{\sigma} },\qquad \text{for all } s \in [|\rho|^2, \infty).
 \end{align*}
Next, to find the $L^p \rightarrow L^q$ estimates for $q\in[2,\infty)$, we will utilize Theorem \ref{thm_spec_mul_int}. We observe that for s fixed $t>0$, the function $s \mapsto e^{-t s^{\sigma}}$  satisfies the hypotheses of Theorem \ref{thm_spec_mul_int}. Thus, applying the same, we obtain 
  \begin{align}\label{est_heat_op}
        \| e^{-t (-\Delta)^{\sigma}} \|_{L^p(\X) \rightarrow L^q(\X)}  & \lesssim  \max\begin{cases} \sup\limits_{| \rho|^2 <s\leq 1+ |\rho|^2}   e^{-t s^{\sigma}}
              \left( s-  |\rho|^2 \right)^{ \frac{\nu}{2}(\frac{1}{p}-\frac{1}{q})},\\
              \sup\limits_{s \geq  1+|\rho|^2}   e^{-t s^{\sigma}} \left( s- |\rho|^2\right)^{\frac{n}{2}(\frac{1}{p}-\frac{1}{q})}  .
          \end{cases}
   \end{align}
First, let us handle the case for $\sigma=1$. By setting $\sigma=1$ and  $\gamma =(s- |\rho|^2 )^{\frac{1}{2}}$, we have from the above 
\begin{align}\label{Lp-q_heat_est}
      \| e^{t \Delta} \|_{L^p(\X) \rightarrow L^q(\X)}  \lesssim  e^{- t|\rho|^{2}}  \max \begin{cases} \sup\limits_{0<\gamma \leq 1}  e^{-t \gamma^{2}}
              \gamma^{\nu(\frac{1}{p}-\frac{1}{q})},\\
             \sup\limits_{\gamma \geq 1}  e^{-t \gamma^{2}} \gamma^{n(\frac{1}{p}-\frac{1}{q})}.
          \end{cases} 
\end{align}
To estimate the supremum of the expression above, let us consider the following functions
\begin{align*}
    \mathcal{I}_{t,\text{glo}}(\gamma):= e^{-t \gamma^{2}} \gamma^{\frac{\nu}{r}},  \quad \text{and} \quad  \mathcal{I}_{t,\text{loc}}(\gamma):= e^{-t \gamma^{2}} \gamma^{\frac{n}{r}} 
\end{align*}
for $ \gamma > 0$, where $1/r=1/p-1/q$. By taking its derivative, we get
\begin{align*}
     \frac{d}{d\gamma}\mathcal{I}_{t,\text{glo}}(\gamma) = e^{-t \gamma^{2}} \left( -  2 t  \gamma^{2}    + \frac{\nu}{r} \right) \gamma^{\frac{\nu}{r}-1}, 
\end{align*}
which has zero only at $\gamma_{\text{glo}}(t) =  ({\nu}/{2  r t})^{\frac{1}{2}}$. Moreover, the function changes sign from positive to negative at $\gamma_{\text{glo}}(t)$. Thus $\gamma_{\text{glo}}(t)$ is a point of maximum of $\mathcal{ I}_{t,\text{glo}}$. Similarly,  $\gamma_{\text{loc}}(t)=({n}/{2 r t})^{\frac{1}{2}} $ is a point of maximum of $\mathcal{ I}_{t,\text{loc}}$ and the function $\gamma \rightarrow \mathcal{ I}_{t,\text{loc}}(\gamma) $ is decreasing on the interval $[\gamma_{\text{loc}}(t),\infty)$.

For large $t$, we have $\gamma_{\text{loc}}(t) \leq 1$ (in particular take $t\geq \nu /2 r$). Utilizing this and the decreasing property of  $\mathcal{ I}_{\text{t,loc}}$ on the interval $[\gamma_{\text{loc}}(t),\infty)$, we get
\begin{align*} 
      \sup\limits_{\gamma \geq 1}  \mathcal{ I}_{\text{t,loc}}(\gamma) \leq \mathcal{ I}_{\text{t,loc}}(1) = e^{-t}. 
\end{align*}
On the other hand, we can write for $t$ near infinity, 
\begin{align*}
 \sup\limits_{0<\gamma \leq 1}  \mathcal{I}_{t,\text{glo}}(\gamma) \leq  \mathcal{I}_{t,\text{glo}}( ({\nu}/{2  r t})^{\frac{1}{2}})=  C t^{-\frac{\nu}{2 } (\frac{1}{p}-\frac{1}{q})}.
\end{align*}
Since for large $t$, $e^{-t} \lesssim t^{-\frac{\nu}{2 } (\frac{1}{p}-\frac{1}{q})},$ so we have derived the following for $t$ near infinity 
\begin{align}\label{est_t>1}
     \max \left\{  \sup\limits_{0<\gamma \leq 1}  e^{-t \gamma^{2}}
              \gamma^{\nu(\frac{1}{p}-\frac{1}{q})}, \, 
             \sup\limits_{\gamma \geq 1}  e^{-t \gamma^{2}} \gamma^{n(\frac{1}{p}-\frac{1}{q})} \right\}\lesssim t^{-\frac{\nu}{2 } (\frac{1}{p}-\frac{1}{q})}.
\end{align}
Now, suppose $t$ is near zero. In this case, we observe
\begin{align*}
     \sup\limits_{0<\gamma \leq 1}  \mathcal{I}_{t,\text{glo}}(\gamma) \leq \mathcal{I}_{t,\text{glo}}(1)=e^{-t}
\end{align*}
and 
\begin{align*}
     \sup\limits_{\gamma \geq 1}  \mathcal{I}_{t,\text{loc}}(\gamma) \leq \mathcal{I}_{t,\text{loc}}(({n}/{2 r t})^{\frac{1}{2}})= C  t^{-\frac{n}{2 }(\frac{1}{p}-\frac{1}{q}) }.
\end{align*}
Hence, for $t$ near zero, we can say the following
\begin{align}\label{est_t<1}
     \max \left\{  \sup\limits_{0<\gamma \leq 1}  e^{-t \gamma^{2}}
              \gamma^{\nu(\frac{1}{p}-\frac{1}{q})}, \, 
             \sup\limits_{\gamma \geq 1}  e^{-t \gamma^{2}} \gamma^{n(\frac{1}{p}-\frac{1}{q})} \right\}\lesssim t^{-\frac{n}{2 } (\frac{1}{p}-\frac{1}{q})}.
\end{align}
In addition,  we note that for  $t$ within a compact interval on $(0,\infty)$, the operator norm $\| e^{t\Delta} \|_{L^p(\X) \rightarrow L^q(\X)}$ can be bounded by a constant independent of $t$. Therefore, substituting the estimates \eqref{est_t>1} and \eqref{est_t<1} in \eqref{Lp-q_heat_est},  the proof of Corollary \ref{cor_est_fr_heat} follows for the heat semigroup. 

Next, we will derive an estimate for the fractional cases. Although the estimate in \eqref{est_heat_op} applies to the $\sigma$-fractional heat semigroup for all $\sigma>0$, due to certain technical challenges, we will need to employ a modified version of Theorem \ref{thm_spec_mul_int}.
\begin{corollary}\label{cor_spec_op_heat}
    Let $\X$ be a general symmetric space of noncompact type, and let $1<p\leq 2 \leq q<\infty$ with $p\not =q$. Assume that $\varphi: (|\rho|^2, \infty)  \rightarrow  \C \setminus \{0\} $ is a function that satisfies all the hypotheses of Theorem \ref{thm_spec_mul_int}. Then the spectral multiplier operator satisfies the following estimates 
\begin{equation}\label{est_spec_op_heat}
      \begin{aligned}
        \|\varphi(-\Delta)\|_{L^p(\X) \rightarrow L^q(\X)}  & \leq C_{p,q}  \sup_{s> |\rho|^2}  \left( \int\limits_{\{  \lambda \in \fa^* \setminus \{0\} :|\lambda| <   \sqrt{s-|\rho|^2}\}}  |\varphi(|\lambda|^2+ |\rho|^2)|^{\frac{pq}{q-p}} |c(\lambda)|^{-2} d\lambda \right)^{\frac{1}{p}-\frac{1}{q}}.
   \end{aligned}
   \end{equation}
\end{corollary}
\begin{proof}
    We can use \eqref{est_of_spec_opn} to get
    \begin{align*}
          &\|\varphi(-\Delta)\|_{L^p(\X) \rightarrow L^q(\X)}  \\ 
          &\lesssim \sup_{|\varphi(s)|< |\varphi(|\rho|^2)|}  |\varphi(s)| \left( \int\limits_{\{  \lambda \in \fa^*  \setminus \{0\}: |\varphi\left({|\lambda|^2+|\rho|^2}\right)|>| \varphi(s)|\}}  |c(\lambda)|^{-2} d\lambda \right)^{\frac{1}{p}-\frac{1}{q}}\\
          &\lesssim \sup_{|\varphi(s)|< |\varphi(|\rho|^2)|}  |\varphi(s)| \left( \int\limits_{\{  \lambda \in \fa^* \setminus \{0\} : |\varphi\left({|\lambda|^2+|\rho|^2}\right)|>| \varphi(s)|\}} \left(\frac{|\varphi(|\lambda|^2+ |\rho|^2)| }{|\varphi(s)|}\right)^{\frac{pq}{q-p}} |c(\lambda)|^{-2} d\lambda \right)^{\frac{1}{p}-\frac{1}{q}}\\
          &\lesssim \sup_{s> |\rho|^2}    \left( \int\limits_{\{  \lambda \in \fa^* \setminus \{0\}:|\lambda| <   \sqrt{s-|\rho|^2}\}}  |\varphi(|\lambda|^2+ |\rho|^2)|^{\frac{pq}{q-p}}|c(\lambda)|^{-2} d\lambda \right)^{\frac{1}{p}-\frac{1}{q}},
    \end{align*}
    completing the proof.
\end{proof}
We recall that $\varphi(s) = e^{-t s^{\sigma}}$ (where $\sigma > 0$) satisfies all the assumptions of Theorem \ref{thm_spec_mul_int}. Therefore, we can apply Corollary \ref{cor_spec_op_heat}. To proceed, let us first assume that $|\rho|^2 < s \leq 1 + |\rho|^2$. In this range of $s$, by setting $1/r = 1/p - 1/q$, that is $r=pq/(q-p)$ and using the estimate \eqref{est_c-2,hr} for $|c(\lambda)|^{-2}$, we obtain the following for all $t \in (0, \infty)$
\begin{align*}
     \int\limits_{\{  \lambda \in \fa^* \setminus \{0\} :|\lambda| <   \sqrt{s-|\rho|^2}\}} |\varphi(|\lambda|^2+ |\rho|^2)|^{\frac{pq}{q-p}} |c(\lambda)|^{-2} d\lambda \lesssim  \int\limits_{0}^{1}  e^{-tr(\gamma^2+|\rho|^2)^{\sigma} } \gamma ^{\nu -1} d\gamma.
\end{align*}
When $s\geq 1+|\rho|^2$, we have for all $t\in (0,\infty)$
\begin{align*}
     &\int\limits_{\{  \lambda \in \fa^* \setminus \{0\} :|\lambda| <   \sqrt{s-|\rho|^2}\}} |\varphi(|\lambda|^2+ |\rho|^2)|^{\frac{pq}{q-p}} ||c(\lambda)|^{-2} d\lambda \\ 
     &\hspace{4cm} \lesssim  \int\limits_{0}^{1}  e^{-tr(\gamma^2+|\rho|^2)^{\sigma} } \gamma ^{\nu -1} d\gamma + \int\limits_{1}^{\infty}  e^{-tr(\gamma^2+|\rho|^2)^{\sigma} } \gamma ^{n -1} d\gamma.
\end{align*}
Now using \cite[Lemma 2]{CGM95},
\begin{align*}
    \int\limits_{0}^{1}  e^{-tr(\gamma^2+|\rho|^2)^{\sigma} } \gamma ^{\nu -1} d\gamma \asymp \begin{cases}
        1  &\quad \text{for all } t\in(0,1],\\
        t^{-\frac{\nu}{2} } e^{-r |\rho|^{2\sigma}} & \quad\text{for all } t\in [1,\infty),
    \end{cases}
\end{align*}
and 
\begin{align*}
    \int\limits_{1}^{\infty}  e^{-tr(\gamma^2+|\rho|^2)^{\sigma} } \gamma ^{n -1} d\gamma \asymp \begin{cases}
        t^{-\frac{n}{2\sigma}}  &\quad \text{for all } t\in(0,1],\\
        t^{-1 } e^{-r (1+|\rho|^2)^{\sigma}} & \quad\text{for all } t\in [1,\infty).
    \end{cases}
\end{align*}
Finally, by considering the supremum on the right-hand side of \eqref{est_spec_op_heat} for $s$ near and away from $|\rho|^2$ (as in \eqref{est_heat_op}) and utilizing the inequalities above, we conclude the proof of Corollary \ref{cor_est_fr_heat}.

\section{Hausdorff-Young-Paley inequality on non-unitary duals}\label{sec_HYP_ND}

In Theorem \ref{thm_HYP_uni}, we proved the Hausdorff-Young-Paley inequality for general symmetric spaces, considering the Fourier transform $\widetilde{f}$ as a function on $\mathfrak{a}^* \times K$. Given that for an $L^p(\X) $ $(1\leq p<2)$ function $f$, $\widetilde{f}$ is holomorphic on a tube containing $\mathfrak{a}^*$, it is of independent interest to determine whether an analogue of the Hausdorff-Young-Paley inequality holds on the entire tube. In this section, we will prove a version of the Hausdorff-Young-Paley inequality on the entire domain where the Fourier transform is well defined in the context of rank one symmetric spaces $\X$ of noncompact type.

First, we establish an analogue of the Paley-type inequality (Theorem \ref{thm_paley_int}) on non-unitary duals, which can be viewed as a weighted version of the Plancherel theorem for $L^p(\X)$ with $1 \leq p < 2$. Unlike in Section \ref{sec_HYP_UD}, we will apply analytic interpolation for mixed norm spaces and rely on both upper and lower estimates of $|c(\lambda)|^{-2}$, as provided in \eqref{sharp_c-2} for the rank one setting.  Additionally, we will present a generalized version of Theorem \ref{thm_paley_int}, which we will utilize for analytically interpolating with the Hausdorff-Young inequality.  To simply the notation, we recall that in the rank one case, $\mathfrak{a}^*$ can be identified with $\mathbb{R}$, and $\mathfrak{a}_{\C}^*$ with the complex plane $\mathbb{C}$. We will prove the following Paley-type inequality.
\begin{theorem}\label{thm_paley_new} Let $\X$ be a rank one symmetric space of noncompact type with dimension $n\geq 1$. Assume $u>0$ be a measurable function on  $\R$ such that  $$\|u\|_{L^1(|c(\lambda)|^{-2})} =  \int_{\R}  u(\lambda)  |c(\lambda)|^{-2}  d\lambda <\infty.$$
 Then for $1\leq p\leq 2$, $p\leq q\leq  p'$, and $\zeta \in \R$, we have 
     \begin{multline} \label{paley's}
        \left( \int_{\R} \left( \int_K  \left|\widetilde {f}(\lambda+\zeta +i\rho_{q} , k) \right|^{q}  \frac{ |\lambda+\zeta+i\rho_{q}|^{q}}{|\lambda+\zeta+i\rho_{q} +i 2\rho|^{q}}  \,  dk \right)^{\frac{p}{q}}  u(\lambda)^{2-p} |c(\lambda)|^{-2}\,d\lambda \right) ^{\frac{1}{p}} \\ \leq C_{p,q} (1+|\zeta|)^{\frac{n-1}{2}} \|u\|_{L^1(|c(\lambda)|^{-2})}^{{\frac{2}{p}-1}}\|f\|_{L^p(\X)},
    \end{multline}
    for all $f\in C_c^{\infty}(\X)$, where $ C_{p,q}$ does not depend on $\zeta$.
\end{theorem}

\begin{proof}  We will utilize the theory of analytic interpolation to establish the theorem above. To start with, consider the measure space $(\X, dx)$ and the product measure space 
\begin{align}\label{defn_Y_nu_1}
(Y, d{\mu_1}):= (\R \times K,   u(\lambda)^{2} \lambda^2 (1+|\lambda|)^{n-3}d\lambda dk).
\end{align}
Let us also define an analytic family of linear operators for compactly supported smooth functions $f$ in $\X$ as follows
    \begin{align}\label{def_psi_z,c}
         \mathcal{S}_{z,\zeta} (f)(\lambda,k) = \frac{\widetilde{f} (\lambda+z+\zeta, k) (\lambda+z+\zeta)}{u(\lambda) (\lambda+z+ \zeta+i 2\rho)} , \quad \text{for all } |\Im z| \leq \rho,
    \end{align}
    where $\zeta $ is an arbitrary but fixed real number. For every $z\in \C$ with $ |\Im z| \leq \rho$, we have  $$|z+i2\rho|\geq |\Im z+ 2\rho| \geq \rho.$$ Consequently, $\lambda +z+i 2\rho$ will never be zero for any $\lambda \in \R$. Therefore, it follows that $\mathcal{S}_{z,\zeta}$ forms an analytic family of linear operators defined on the strip $0 < |\Im z| < \rho$.

Now, let $z =\xi +i\rho$, then we can express
    \begin{align*}
        \| \mathcal{S}_{\xi +i\rho,\zeta} (f)\|_{(1,1)} &= \int_{\R} \int_{K}  \left| \widetilde{f} (\lambda+\xi+\zeta+ i\rho, k) \right|  \frac {|\lambda+\xi+\zeta+i\rho|} {|\lambda+\xi+\zeta+i3\rho|}   u(\lambda) \lambda^{2} (1+|\lambda|)^{n-3}\,  d\lambda\, dk.
    \end{align*}
Here, we apply the restriction theorem for symmetric spaces. More precisely, by applying Fubini's theorem followed by Theorem \ref{RS_4.2}, we obtain
    \begin{align*}
         \|   \mathcal{S}_{\xi +i\rho,\zeta} (f) \|_{(1,1)} &= \int_{\R}  \left( \int_{K}   | \widetilde{f} (\lambda+\xi+\zeta+i\rho, k) | \, dk \right)  \frac {|\lambda+\xi+\zeta+i\rho|} {|\lambda+\xi+\zeta+i3\rho|}  u(\lambda) \lambda^{2} (1+|\lambda|)^{n-3} \,d\lambda\\
         &  \leq C \|f\|_{L^1(\X)}  \int_{\R}   u(\lambda) \lambda^2 (1+|\lambda|)^{n-3}\, d\lambda,
    \end{align*}
    where in the last step, we used the fact that for $\lambda \in \R$,
    ${|\lambda+i\rho|} /{|\lambda+i3\rho|} \leq 1.$
        Using the sharp estimate \eqref{sharp_c-2} of $|c(\lambda)|^{-2}$, we get 
    \begin{align}\label{Tz_1,1}
         \|\mathcal{S}_{\xi +i\rho,\zeta} (f) \|_{(1,1)} & \leq C \|f\|_{L^1(\X)}  \|u\|_{L^1(|c(\lambda)|^{-2})}.
    \end{align}
    For $z = \xi-i \rho$, $\xi\in \R$, we will prove the following 
      \begin{align}\label{Tz_inf,1}
         \|\mathcal{S}_{\xi -i\rho,\zeta} (f) \|_{(\infty,1)} & \leq C \|f\|_{L^1(\X)}  \|u\|_{L^1(|c(\lambda)|^{-2})}.
    \end{align}
    In fact, for $z = \xi-i \rho$, we can write for all $k\in K$
    \begin{align*}
       |\mathcal{S}_{\xi-i \rho, \zeta} (f)(\lambda, k)|& \leq   C \left| \int_{G} f(g) e^{\left(i(\lambda+\xi-i\rho)- \rho \right)H(g^{-1}k)} dg\right| \left|\frac{\lambda+\xi+\zeta-i \rho}{\lambda+\xi+\zeta+i \rho} \right|\\
    & \leq  \int_{G} |f(g)| \, dg \left|\frac{\lambda+\xi+\zeta-i \rho}{\lambda+\xi+\zeta+i \rho} \right|.
    \end{align*}
    Now, integrating both sides with respect to the measure $  u(\lambda) \lambda^2 (1+|\lambda|)^{n-3} d\lambda$ and utilizing the fact that for $\lambda \in \R$, $|\lambda-i \rho|/|\lambda+i \rho|\leq  1$, we get
    \begin{align*}
         \|\mathcal{S}_{\xi -i\rho,\zeta} (f) \|_{(\infty,1)}\leq C \|f\|_{L^1(\X)}  \int_{\R}   u(\lambda) \lambda^2 (1+|\lambda|)^{n-3}\, d\lambda =C \|f\|_{L^1(\X)}\|u\|_{L^1(|c(\lambda)|^{-2})}.
    \end{align*}
      Next, for $z =\xi \in \R$, we have from the definition  \eqref{def_psi_z,c} that
    \begin{align*}
        \|\mathcal{S}_{\xi,\zeta} f\|^2_{(2,2)} &= \int_{\R} \int_{K} \left| \widetilde{f} (\lambda+\xi+\zeta, k) \right|^2  \frac {|\lambda+\xi+\zeta|^2} {u(\lambda)^2|\lambda+\xi+\zeta+i2\rho|^2}{u(\lambda)^2 } \lambda^2 (1+|\lambda|)^{n-3}\,  d\lambda \,dk \\
        &\leq  \int_{K}  \int_{\R}  \left| \widetilde{f} (\lambda+\xi+\zeta, k) \right|^2  \frac {|\lambda+\xi+\zeta|^2} {|\lambda+\xi+\zeta+i2\rho|^2}  (1+|\lambda|)^{n-1}  \,  d\lambda  \,dk.
    \end{align*}
   Changing the variable to $\lambda \rightarrow \lambda-\xi-\zeta$, the inequality above  transforms into
    \begin{align*}
        \|\mathcal{S}_{\xi,\zeta} f\|^2_{(2,2)} &\leq  \int_{\R} \int_{K}  \left| \widetilde{f} (\lambda, k) \right|^2  \frac {\lambda^2} {|\lambda+i2\rho|^2}  (1+|\lambda-\xi-\zeta|)^{n-1}\,  d\lambda \,dk\\
        & \leq  \int_{\R} \int_{K}  \left| \widetilde{f} (\lambda, k) \right|^2  \frac {\lambda^2} {|\lambda+i2\rho|^2}  (1+|\lambda|+|\xi|+|\zeta|)^{n-1}\,  d\lambda \,dk.        
    \end{align*}
 As $\rho$ is always greater or equal to $1/2$, then utilizing the  following for $\lambda \in \R$
\begin{align*}
|\lambda+i2\rho|^2\asymp (1+|\lambda|)^2 \qquad \text{and} \qquad (1+|\lambda|+|\xi|+|\zeta|) \leq (1+|\lambda|) (1+|\xi|)(1+|\zeta|),
\end{align*}
 we can write
    \begin{align*}
      \|\mathcal{S}_{\xi,\zeta} f\|^2_{(2,2)} &\leq C (1+|\xi|)^{n-1}(1+|\zeta|)^{n-1} \int_{\R} \int_{K}  \left| \widetilde{f} (\lambda, k) \right|^2 \lambda^2 (1+|\lambda|)^{n-3} \, d\lambda\, dk. 
    \end{align*}
So, from the estimate \eqref{sharp_c-2} and Plancherel theorem \eqref{Planc}, we have 
\begin{align}\label{Tz_2,2_P}
     \|\mathcal{S}_{\xi,\zeta} f\|_{(2,2)} &\leq C (1+|\xi|)^{\frac{n-1}{2}}(1+|\zeta|)^{\frac{n-1}{2}}  \|f\|_{L^2(\X)}.
\end{align}
By employing analytic interpolation for mixed norm spaces (see \eqref{eqn_ana_int}), we first derive from \eqref{Tz_1,1} and \eqref{Tz_2,2_P} for any $1\leq p\leq 2$, 
\begin{align*}
     \| \mathcal{S}_{i\rho_p, \zeta}(f)\|_{(p,p)} \leq C_p (1+|\zeta|)^{\frac{n-1}{2}} \|u\|_{L^1(|c(\lambda)|^{-2})}^{\frac{2}{p}-1} \|f\|_{L^p(\X)},
\end{align*}
which, by substituting the definition of mixed norm space, transforms into 
\begin{multline}\label{Psi_p,p}
    \left(\int_{\mathbb{R}} \int_K \left|\widetilde{f}\left(\lambda+\zeta+i \rho_p, k\right) \right|^p \frac{|\lambda+\zeta+ i \rho_p|^{p}}{|\lambda+\zeta+i\rho_p+i2\rho|^p} \, u(\lambda)^{2-p} \, |c(\lambda)|^{-2}  d k \, d \lambda\right)^{\frac{1}{p}}  \\  \leq C_p (1+|\zeta|)^{\frac{n-1}{2}} \|u\|_{L^1(|c(\lambda)|^{-2})}^{\frac{2}{p}-1} \|f\|_{L^p(\X)}.
\end{multline}
Similarly, by analytically interpolating again, this time with  \eqref{Tz_inf,1} and \eqref{Tz_2,2_P}, we obtain
\begin{align*}
     \| \mathcal{S}_{i\rho_{p'}, \zeta}(f)\|_{(p',p)} \leq C_p (1+|\zeta|)^{\frac{n-1}{2}} \|u\|_{L^1(|c(\lambda)|^{-2})}^{\frac{2}{p}-1} \|f\|_{L^p(\X)}.
\end{align*}
Substituting the definition of mixed norm space, we can rewrite the inequality above as:
\begin{multline}\label{Psi_p',p}
    \left(\int_{\mathbb{R}}\left( \int_K\left|\widetilde{f}\left(\lambda+\zeta+i \rho_{p'}, k\right)\right|^{p'} \frac{|\lambda+\zeta+i \rho_{p'}|^{p'}}{|\lambda+\zeta+i\rho_{p'}+i2\rho|^{p'}}  \, d k \right)^{\frac{p}{p'}}\, u(\lambda)^{2-p} |c(\lambda)|^{-2} \, d\lambda\right)^{\frac{1}{p}}\\ \leq C_p (1+|\zeta|)^{\frac{n-1}{2}} \|u\|_{L^1(|c(\lambda)|^{-2})}^{\frac{2}{p}-1} \|f\|_{L^p(\X)}.
\end{multline}

Thus, it follows from \eqref{Psi_p,p} and \eqref{Psi_p',p} that we have established the theorem for the cases $q=p$ and $q=p'$. Now, we turn our attention to the cases where $p<q<p'$. To tackle this, we will utilize real interpolation theory and consider the same measure spaces $(\X, dx)$ and $(Y,d{\mu_{1}})$ as before. Let us define a family of linear operator from $\X$ to measurable functions on $Y$ for any $q \in [1,\infty]$
\begin{align}
    \mathcal{T}_{q,\zeta}(f)(\lambda,k)= \frac{\widetilde{f} (\lambda+\zeta+ i\rho_q, k) (\lambda+\zeta+i\rho_q)}{u(\lambda)(\lambda+\zeta+i\rho_q+i 2\rho)},
\end{align}
where $\zeta$ is an arbitrary but fixed real number. Since \eqref{Psi_p,p} and \eqref{Psi_p',p} hold true for all $p\in [1,2]$, by taking $q \in (p,2]$ and replacing $p$ with $q$ in the inequalities above, we get
\begin{align}\label{T_q,q_P}
    \|\mathcal{T}_{q,\zeta}(f)\|_{(q,q)} \leq C_q   (1+|\zeta|)^{\frac{n-1}{2}} \|u\|_{L^1(|c(\lambda)|^{-2})}^{\frac{2}{q}-1} \|f\|_{L^q(\X)}.
\end{align}
and 
\begin{align}\label{T_q',q_P}
    \|\mathcal{T}_{q',\zeta}(f)\|_{(q',q)} \leq  C_q   (1+|\zeta|)^{\frac{n-1}{2}} \|u\|_{L^1(|c(\lambda)|^{-2})}^{\frac{2}{q}-1} \|f\|_{L^q(\X)}.
\end{align}
Additionally, we will prove that the following inequalities hold for any $q\in [1,2]$:
\begin{align}\label{T_q,1_p}
\|\mathcal{T}_{q,\zeta}(f)\|_{(q, 1)} \leq \|u\|_{L^1(|c(\lambda)|^{-2})} \|f\|_{L^1({\X})}
\end{align}
and
\begin{align}\label{T_q',1_p}
\|\mathcal{T}_{q',\zeta}(f)\|_{(q',1)} \leq \|u\|_{L^1(|c(\lambda)|^{-2})} \|f\|_{L^{1}({\X})}.
\end{align}
To prove the inequalities above, we write  
\begin{align*}
  \|\mathcal{T}_{q,\zeta}(f)\|_{(q, 1)} &= \int_{\R}\left( \int_K|\mathcal{T}_{q,\zeta}(f) (\lambda,k)|^q dk \right)^{\frac{1}{q}}  u(\lambda)^2 \lambda^2 (1+|\lambda|)^{n-3} d\lambda  \\ 
  & = \int_{\R} \left(\int_K \left|\widetilde{f}\left(\lambda+\zeta+i \rho_q , k \right)\right|^q \frac{|\lambda+\zeta+i\rho_q|^q}{|\lambda+\zeta+i\rho_q+i 2\rho|^q}\, d k\right)^{\frac{1}{q}}  u(\lambda) \lambda^2 (1+|\lambda|)^{n-3} d\lambda .  
    \end{align*}
Utilizing the estimates $|\lambda+i\rho_q|/|\lambda+i\rho_q+i 2\rho| \leq 1$ for $\lambda \in \R$, \eqref{sharp_c-2}, and recalling that $\rho_q =(2/q-1)\rho$,  we deduce from the inequality above that
\begin{align*}
    \|\mathcal{T}_{q,\zeta}(f)\|_{(q, 1)} &   \leq \int_{\R} \left(\int_K\left(\int_G\left|f(g)  e^{((i\lambda -(\frac{2}{q}-1) \rho -\rho)H(g^{-1}k)}\right|\, d g\right)^q \,d k\right)^{\frac{1}{q}} u(\lambda) |c(\lambda)|^{-2} d\lambda\\
    & \leq \|u\|_{L^1(|c(\lambda)|^{-2})}  \int_G |f(g)|\left(\int_K e^{-2 \rho H(g^{-1}k)} d k\right)^{\frac{1}{q}}\, d g,
\end{align*}
where in the last step we used Minkowski's integral inequality. Now we apply \eqref{F_C(K)=1} in the inequality above to obtain \eqref{T_q,1_p}. The proof of \eqref{T_q',1_p} follows similarly.

Finally, using the Riesz-Thorin interpolation theorem for mixed norm spaces between \eqref{T_q,q_P} and \eqref{T_q,1_p}, we obtain
 \begin{align}\label{T_q,p_P}
    \|\mathcal{T}_{q,\zeta}(f)\|_{(q,p)} \leq C_{p,q}  (1+|\zeta|)^{\frac{n-1}{2}} \|u\|_{L^1(|c(\lambda)|^{-2})}^{\frac{2}{p}-1} \|f\|_{L^p(\X)}
 \end{align}
Similarly, from \eqref{T_q',q_P} and \eqref{T_q',1_p}, we have
  \begin{align}\label{T_q',p_P}
    \|\mathcal{T}_{q,\zeta}(f)\|_{(q',p)} \leq C_{p,q}  (1+|\zeta|)^{\frac{n-1}{2}} \|u\|_{L^1(|c(\lambda)|^{-2})}^{\frac{2}{p}-1} \|f\|_{L^p(\X)}
 \end{align}
 for all $q \in (p,p')$ and $p\in [1,2]$. 
 
 Therefore, by considering the inequalities \eqref{T_q,p_P}, \eqref{T_q',p_P}, as well as \eqref{Psi_p,p} and \eqref{Psi_p',p}, we can conclude our theorem.
\end{proof}
\begin{remark} 
\begin{enumerate}
    \item It is natural to ask whether Theorem \ref{thm_paley_int} holds if we assume $\|u\|_{c,\infty}$ is finite, as we did in Theorem \ref{thm_paley_uni}, instead of $\|u\|_ {L^{1}(|c(\lambda)|^{-2})}<\infty$. We note that if $\|u\|_{c,\infty} < \infty$, then we can prove that the operator $\mathcal{S}_{i\rho,0}$, defined in \eqref{def_psi_z,c}, is of weak type $(1,1)$ (Theorem \ref{thm_paley_uni}). However, it may not be of strong type $(1,1)$, which seems necessary to apply Stein's analytic interpolation.
    \item If $q\not =2$, then the subsequent inequality follows from Theorem \ref{thm_paley_new}:  for any $p\leq q\leq  p'$, 
     \begin{align*}
        \left( \int_{\fa^*} \left( \int_K  \left|\widetilde {f}(\lambda +i\rho_{q} , k) \right|^{q}   \,  dk \right)^{\frac{p}{q}}  u(\lambda)^{2-p} |c(\lambda)|^{-2} d\lambda \right) ^{\frac{1}{p}} \leq C_{p,q} \|u\|_{L^1(|c(\lambda)|^{-2})}^{{\frac{2}{p}-1}} \|f\|_{L^p(\X)},
    \end{align*}
    for all $f\in L^p(\X)$. For $q=2$, the inequality above follows from Theorem \ref{thm_paley_uni}, by assuming the weaker condition $ \|u\|_{c,\infty} < \infty$ on $u$. 
    \item By setting $q=p$ in Theorem \ref{thm_paley_int}, we recover \cite[Theorem 1.7]{KPRS24} as a special case. When assuming that $\|u\|_ {L^{1}(|c(\lambda)|^{-2})}<\infty$, one can apply restriction theorems from \cite{RS_09, KRS10} to derive additional Fourier inequalities, as discussed in \cite[Remark 6.1]{KPRS24}. 
\end{enumerate}
\end{remark}
Our next goal is to prove the Hausdorff-Young-Paley inequality on non-unitary duals. To achieve this, we require a modified version of the Hausdorff-Young inequality compared to \cite[Theorem 4.6]{RS_09}, which is crucial for our proof. While we follow the Ray-Sarkar approach and use Stein's analytic interpolation, we have defined a different family of analytic linear operators following \cite{CGM93}.  This adaptation has enabled us to establish the Hausdorff-Young inequality specifically for the case when $n=2$, an aspect not covered in \cite[Theorem 4.6]{RS_09}.
\begin{proof}[\textbf{Proof of Theorem \ref{thm_H-Y_int}}]
Let $p\in [1,2)$ be fixed. We begin by considering the following measure spaces: $(\X, dx)$ and the product measure space \begin{align}\label{defn_Y,nu_0}
(Y, d{\mu_{0}}):= (\R \times K,   (1+|\lambda|)^{n-1}d\lambda dk).
\end{align} 
We define an analytic family of linear operators for compactly supported smooth functions $f$ in $\X$ to measurable functions on $Y$ by
    \begin{align}\label{defn_Tz_HYP}
        \mathcal{S}_z (f)(\lambda, k) = \Tilde{f} (\lambda+z, k) \frac{(\lambda+z)}{(\lambda+z+i 2\rho)} , \quad \text{where } 0 \leq  |\Im z| \leq \rho.
    \end{align}
 Following the same argument as for \eqref{def_psi_z,c}, $\mathcal{S}_z$ is an analytic family of linear operators defined on the strip $0 < |\Im z| < \rho$.

\noindent For $z=\xi \in \mathbb{R}$,  we can write using \eqref{defn_Tz_HYP},
\begin{equation}
\begin{aligned}
\left\|\mathcal{S}_{\xi} (f)\right\|_{ (2,2)}^2 & =  \int_{\mathbb{R}} \int_K|\widetilde{f}(\lambda+\xi, k)|^2 \frac{|\lambda+\xi|^2}{|\lambda+\xi+i 2\rho|^2}(1+|\lambda|)^{n-1} \,d \lambda\, dk .
\end{aligned}
\end{equation}
Next, by the change of variable $\lambda \rightarrow \lambda -\xi$, the above equation simplifies to
\begin{align*}
    \left\|\mathcal{S}_{\xi} (f)\right\|_{(2,2)}^2 & = \int_{\mathbb{R}} \int_K|\widetilde{f}(\lambda, k)|^2 \frac{|\lambda|^2}{|\lambda+i 2\rho|^2}(1+|\lambda-\xi|)^{n-1} \,d \lambda\, dk\\
    & \leq  \int_{\mathbb{R}} \int_K|\widetilde{f}(\lambda, k)|^2 \frac{|\lambda|^2}{|\lambda+i 2\rho|^2}(1+|\lambda|+|\xi|)^{n-1} \,d \lambda\, dk\\
    & \leq   (1+|\xi|)^{n-1} \int_{\mathbb{R}} \int_K|\widetilde{f}(\lambda, k)|^2 \frac{|\lambda|^2}{|\lambda+i 2\rho|^2}(1+|\lambda|)^{n-1} \,d \lambda\, dk.
\end{align*}
By employing the sharp estimate of $|c(\lambda)|^{-2}$ from Lemma \ref{est_c-2}, followed by the Plancherel theorem, we derive from the above
\begin{align*}
    \left\|\mathcal{S}_{\xi} (f)\right\|_{(2,2)}^2 &  \leq C (1+|\xi|)^{n-1} \int_{\mathbb{R}} \int_K|\widetilde{f}(\lambda, k)|^2 {|\lambda|^2}(1+|\lambda|)^{n-3} \,d \lambda\, dk\\
    & \leq C  (1+|\xi|)^{n-1} \int_{\mathbb{R}} \int_K|\widetilde{f}(\lambda, k)|^2  |c(\lambda)|^{-2} \,d \lambda\, dk\\
    &= C (1+|\xi|)^{n-1} \|f\|^2_{L^2(\X)}.
\end{align*}
Therefore, we have proved that there exist a constant $C>0$ such that for all $\xi\in \R$, the following holds
\begin{align}\label{Psi_2,2}
    \left\|\mathcal{S}_{\xi} (f)\right\|_{(2,2)} \leq C (1+|\xi|)^{\frac{n-1}{2}} \|f\|_{L^2(\X)}.
\end{align}
Next, for $z=\xi+i \rho, \xi \in \mathbb{R}$, we have from \eqref{defn_Tz_HYP} that
\begin{equation}
\begin{aligned}
\int_K\left|\mathcal{S}_{\xi+i \rho} (f)(\lambda, k)\right| \,d k & \leq \int_K|\widetilde{f}(\xi+i \rho+\lambda, k)| \left|\frac{\lambda+\xi+i \rho}{\lambda+\xi+i 3\rho} \right| \,d k.
\end{aligned}
\end{equation}
Since for all $\lambda \in \R$, $|{\lambda+i \rho}| / |{\lambda+i 3\rho}| \leq 1$, we can utilize this inequality along with the definition \eqref{defn:hft} of $\widetilde{f}$ to write
\begin{align*}
    \int_K\left|\mathcal{S}_{\xi+i \rho} (f)(\lambda, k)\right| \,d k & \leq \int_K \left| \int_{G} f(g) e^{(i(\lambda+\xi+i\rho)- \rho)H(g^{-1}k)} dg\right| dk\\
    & \leq  \int_{G}  |f(g)| \left( \int_K e^{-2\rho H(g^{-1}k)}\, dk\,\right)  dg\\
    & = \int_{G}  |f(g)|  dg,
\end{align*}
where in the last step we used \eqref{F_C(K)=1}. Thus, we can say
\begin{align}\label{Psi_1,inf}
    \| \mathcal{S}_{\xi+i\rho }(f)\|_{(1,\infty)} \leq    \|f\|_{L^1(\X)}.
\end{align}
Now, by setting $z=\xi- i \rho$ and following a similar argument as before, we can derive the following
\begin{align*}
   |\mathcal{S}_{\xi-i \rho} (f)(\lambda, k)|& \leq   C \left| \int_{G} f(g) e^{\left(i(\lambda+\xi-i\rho)- \rho \right)H(g^{-1}k)} dg\right| \left|\frac{\lambda+\xi-i \rho}{\lambda+\xi+i \rho} \right|\\
    & \leq  \int_{G} |f(g)| \, dg.
\end{align*}
Therefore, we have  the following inequality
\begin{align}\label{Psi_inf,inf}
    \| \mathcal{S}_{\xi-i \rho}(f)\|_{(\infty,\infty)} \leq    \|f\|_{L^1(\X)}.
\end{align}
By employing analytic interpolation for mixed norm spaces, we first derive from \eqref{Psi_2,2} and \eqref{Psi_1,inf} for any $1\leq p\leq 2$, 
\begin{align*}
     \| \mathcal{S}_{i\rho_p}(f)\|_{(p,p')} \leq C_p \|f\|_{L^p(\X)},
\end{align*}
which, by substituting the definition of mixed norm space, transforms into 
\begin{align}\label{Psi_p,p'}
    \left(\int_{\mathbb{R}}\left(\int_K\left|\widetilde{f}\left(\lambda+i \rho_p, k\right)\right|^p \frac{|\lambda+ i \rho_p|^{p}}{|\lambda+i\rho_p+i2\rho|^p} \, d k\right)^{\frac{p^{\prime}}{p}} (1+|\lambda|)^{n-1}  d \lambda\right)^{\frac{1}{p^{\prime}}} \leq C_p \|f\|_{L^p(\X)}.
\end{align}
Similarly, by analytically interpolating again, this time with \eqref{Psi_2,2} and \eqref{Psi_inf,inf}, we obtain
\begin{align}
     \| \mathcal{S}_{i\rho_{p'}}(f)\|_{(p',p')} \leq C_p\|f\|_{L^p(\X)}.
\end{align}
Substituting the definition of mixed norm space, we can express the inequality above as
\begin{align}\label{Psi_p',p'}
\left(\int_{\mathbb{R}}\int_K\left|\widetilde{f}\left(\lambda+i \rho_{p'}, k\right)\right|^{p'} \frac{|\lambda+i \rho_{p'}|^{p'}}{|\lambda+i\rho_{p'}+i2\rho|^{p'}}  (1+|\lambda|)^{n-1} \, d k\, d \lambda\right)^{\frac{1}{p^{\prime}}} \leq C_p \|f\|_{L^p(\X)}.
\end{align}
Therefore, from \eqref{Psi_p,p'} and \eqref{Psi_p',p'}, we can conclude that the theorem is proven for $q=p$ and $q=p'$ respectively. Next, to prove the theorem for $p<q<p'$, we can define 
\begin{align}
    \mathcal{T}_q(f)(\lambda,k)= \widetilde{f} (\lambda+i\rho_q, k) \frac{(\lambda+i\rho_q)}{(\lambda+i\rho_q+i 2\rho)}, \quad \lambda\in \R, k\in K,
\end{align}
and follow a similar calculation as in \cite[(4.20)]{RS_09}.
\end{proof}
In the following, we explore the relation between our Hausdorff-Young inequality presented in Theorem \ref{thm_H-Y_int} and Ray-Sarkar's result Theorem \ref{thm_RS_HYinq}.
\begin{corollary}\label{cor_Our_HY>}
     Let $p\in [1,2)$. Then from \eqref{sharp_c-2} the sharp estimate of $|c(\lambda)|^{-2}$, we observe that following inequality holds for any $q \in [p,p']\setminus\{2\}$
   \begin{align*}
        |c(\lambda)|^{-2} \asymp \lambda^2 (1+|\lambda|)^{n-3} \lesssim  \frac{ |\lambda+i\rho_q|^{p'}}{|\lambda+i\rho_q +i 2\rho|^{p'}} (1+|\lambda|)^{n-1},
    \end{align*}
    since  in this case $  { |\lambda+i\rho_q|}/{|\lambda+i\rho_q +i 2\rho|} \asymp 1$. So, for all  $q \in [p,p']\setminus\{2\}$, we can  utilize the inequality above to obtain
    \begin{align*}
    & \left( \int_{\R} \left( \int_K  \left|\widetilde {f}(\lambda +i\rho_q , k) \right|^q \,  dk \right)^{\frac{p'}{q}}  |c(\lambda)|^{-2} d\lambda \right) ^{\frac{1}{p'}}\\
      &  \leq C \left( \int_{\R} \left( \int_K  \left|\widetilde {f}(\lambda +i\rho_q , k) \right|^q \frac{ |\lambda+i\rho_q|^{q}}{|\lambda+i\rho_q +i 2\rho|^{q}}  \,  dk \right)^{\frac{p'}{q}}  (1+|\lambda|)^{n-1} d\lambda \right) ^{\frac{1}{p'}}
      \leq C\|f\|_{L^p(\X)},
    \end{align*}
    where in the last inequality, we used our result Theorem \ref{thm_H-Y_int}. Thus, as a consequence we obtain the Hausdorff-Young inequality  by Ray and Sarkar (Theorem \ref{thm_RS_HYinq}).  Interestingly we observe that for the $q=2$ case, Theorem \ref{thm_H-Y_int} follows as a corollary of Theorem \ref{thm_RS_HYinq}. 
    \end{corollary}
We will conclude this section by proving (Theorem \ref{thm_H-Y-P}) an analog of the Hausdorff–Young –Paley inequality on non-unitary duals using an analytic version of the Stein-Weiss interpolation theorem.
\begin{proof}[\textbf{Proof of Theorem \ref{thm_H-Y-P}}]
Let $1\leq p\leq 2$, and $p\leq q_0, q_1\leq  p'$. We define the following family of analytic linear operators for compactly supported smooth functions $f$ in $\X$
    \begin{align}
        \mathcal{S}_z (f)(\lambda, k) = \Tilde{f} (\lambda+z, k) \frac{(\lambda+z)}{(\lambda+z+i 2\rho)} , \quad \text{where } 0 \leq  |\Im z| \leq \rho,
    \end{align}
    for all $\lambda \in \R, k\in K$. Then, according to Theorem \ref{thm_H-Y_int}, we can conclude that $\mathcal{S}_{i\rho_{q_0}}$ is a bounded operator from $L^p(\X)$ to $L^{(q_1,p')}(Y, d{\mu_0})$, where we recall $(Y, d{\mu_0})$ from \eqref{defn_Y,nu_0}. Furthermore, for any arbitrary but fixed $\zeta \in \R$, we can write
    \begin{align*}
         & \|\mathcal{S}_{\zeta+i \rho_{q_1}} (f)\|_{L^{(q_1,p')}(Y,{ {d\mu_0}})}^{p'}\\
         &= C     \int_{\R} \left( \int_K  \left|\widetilde {f}(\lambda+\zeta +i\rho_{q_1} , k) \right|^{q_1}  \frac{ |\lambda+\zeta+ i\rho_{q_1}|^{q_1}}{|\lambda+\zeta+i\rho_{q_1} +i 2\rho|^{q_1}}  \,  dk \right)^{\frac{p'}{q_1}}  (1+|\lambda|)^{n-1} d\lambda.
    \end{align*} 
    By the change of variable  $\lambda \rightarrow \lambda-\zeta$ in the integral above, we get
   \begin{align*}
       & \|\mathcal{S}_{\zeta+i \rho_{q_1}} (f)\|_{L^{(q_1,p')}(Y,{ {d\mu_0}})}^{p'} \\
       &= C    \int_{\R} \left( \int_K  \left|\widetilde {f}(\lambda +i\rho_{q_1} , k) \right|^{q_1}  \frac{ |\lambda+ i\rho_{q_1}|^{q_1}}{|\lambda+i\rho_{q_1} +i 2\rho|^{q_1}}  \,  dk \right)^{\frac{p'}{q_1}}  (1+|\lambda-\zeta|)^{n-1} d\lambda.
   \end{align*}
   Using the inequality \begin{align*}
       (1+|\lambda-\zeta|)\leq (1+|\lambda|)(1+|\zeta|),
   \end{align*} we obtain
    \begin{align*}
       & \|\mathcal{S}_{\zeta+i \rho_{q_1}} (f)\|_{L^{(q_1,p')}(Y,{ {d\mu_0}})}^{p'} \\
        & \leq C (1+|\zeta|)^{n-1} \int_{\R} \left( \int_K  \left|\widetilde {f}(\lambda +i\rho_{q_1} , k) \right|^{q_1}  \frac{ |\lambda+ i\rho_{q_1}|^{q_1}}{|\lambda+i\rho_{q_1} +i 2\rho|^{q_1}}  \,  dk \right)^{\frac{p'}{q_1}} (1+|\lambda|)^{n-1}  d\lambda.
       \end{align*}   
   Here we apply Theorem \ref{thm_H-Y_int}, to derive
  \begin{align}\label{Psi_q1_p'}
       \|\mathcal{S}_{\zeta+i \rho_{q_1}} (f)\|_{L^{(q_1,p')}(Y,{ {d\mu_0}})}  \leq C_{p,q_1}  (1+|\zeta|)^{\frac{n-1}{p}}   \|f\|_{p}.
  \end{align}
Next, assuming the function $u$ satisfies condition \eqref{eqn_int_u}, we can apply Theorem \ref{thm_paley_new}. But before delving into that, let us define the following product measure space for each fixed $p$
 \begin{align}
     (Y, d {\mu_2})  = (\R \times K, u(\lambda)^{2-p} |c(\lambda)|^{-2} d\lambda  dk). 
 \end{align}
Consequently, we have from  \eqref{sharp_c-2} and  Theorem \ref{thm_paley_new}  
 \begin{equation}\label{Psi_q0_p}
\begin{aligned}
    & \|\mathcal{S}_{\zeta+i \rho_{q_0}} (f)\|_{L^{(q_0,p)}{(Y,d\mu_2)}}\\
    &\asymp \left( \int_{\R} \left( \int_K  \left|\widetilde {f}(\lambda+\zeta +i\rho_{q_0} , k) \right|^{q_0}  \frac{ |\lambda+\zeta+i\rho_{q_0}|^{q_0}}{|\lambda+\zeta+i\rho_{q_0} +i 2\rho|^{q_0}}  \,  dk \right)^{\frac{p}{q_0}}  u(\lambda)^{2-p} \lambda^2 (1+|\lambda|)^{n-3} d\lambda \right) ^{\frac{1}{p}} \\
    &\leq C_{p,q_1} (1+|\zeta|)^{\frac{n-1}{2}} \|u\|_{L^1(|c(\lambda)|^{-2})}^{{\frac{2}{p}-1}}\|f\|_{L^p(\X)},
\end{aligned}
\end{equation}
Therefore, by utilizing the inequality \eqref{Psi_q1_p'} and \eqref{Psi_q0_p}, and applying Corollary \ref{cor_sw_ana} with $Q_0=(q_0, p)$, $Q_1=(q_1,p')$, and
      \begin{align*}
      {w}_0(\lambda)= u(\lambda)^{2-p} \lambda^2 (1+|\lambda|)^{n-3}, \qquad w_1(\lambda)= (1+|\lambda|)^{n-1},
    \end{align*} 
    we can derive that
    \begin{align}\label{eqn_af_analytic}
        \left( \int_{\R} \left( \int_K  \left|\widetilde {f}(\lambda+i\rho_{q_\theta} , k) \right|^{q_{\theta}}  \frac{ |\lambda+i\rho_{q_{\theta}}|^{q_{\theta}}}{|\lambda+i\rho_{q_\theta} +i 2\rho|^{q_{\theta}}}  \,  dk \right)^{\frac{b}{q_{\theta}}}   w_{\theta}(\lambda) d\lambda \right) ^{\frac{1}{b}}\leq C_{b,q_{\theta}}  {\|u\|_{L^1(|c(\lambda)|^{-2})}^{{(\frac{2}{p}-1)}(1-\theta)}} \|f\|_{L^p(\X)},
    \end{align}
    where 
    \begin{align*}
       \frac{1}{q_{\theta}}= \frac{1-\theta}{q_0}+\frac{\theta}{q_1}, \qquad \frac{1}{b}= \frac{1-\theta}{p}+\frac{\theta}{p'}, \qquad \text{and} \qquad  w_{\theta}(\lambda)=  {w}_0(\lambda)^{\frac{b(1-\theta)}{p}} {w}_1(\lambda)^{\frac{b\theta}{p'}}.
    \end{align*}
    Solving the expression above for $\theta\in (0,1)$, we get $\theta= \frac{b-p}{b(2-p)}$, which implies
    \begin{align*}
          \frac{b(1-\theta)}{p}= \frac{b-bp+p}{p(2-p)}= \frac{p'-b}{p'(2-p)}, \qquad \text{and}\qquad \frac{b \theta}{p'}=  \frac{b-p}{p'(2-p)}
    \end{align*}
   Putting the expressions above in $w_{\theta}$, we obtain
    \begin{align*}
        w_{\theta}(\lambda)= u(\lambda)^{1-\frac{b}{p'}} \lambda^{  \frac{2(p'-b)}{p'(2-p)}} (1+|\lambda|)^{n-3+\frac{2(b-p)}{p'(2-p)}} .
    \end{align*}      
    Finally, substituting the expression above of $w_{\theta}$ in \eqref{eqn_af_analytic}, we conclude our theorem. 
\end{proof}
    \begin{proof}[\textbf{Proof of Corollary \ref{cor_of_HYP}}] We first note that for a given $p \in (1,2]$, we have $\frac{2(p'-b)}{p'(2-p)} \leq 2$ for all $b\in [p,p'] $ and $\frac{2(p'-p)}{p'(2-p)} \geq 2$.  When $\lambda$ near zero, specifically if $|\lambda| \leq 1$, using  $\frac{2(p'-b)}{p'(2-p)} \leq 2$, we have
    \begin{align*}
        |c(\lambda)|^{-2} \lesssim |\lambda|^{2} \lesssim |\lambda|^{\frac{2(p'-b)}{p'(2-p)} }.
    \end{align*}
    Similarly, when $|\lambda|>1$ using $\frac{2(p'-p)}{p'(2-p)} \geq 2$,  we can write 
    \begin{align*}
        |c(\lambda)|^{-2} \lesssim (1+|\lambda|)^{n-1} \lesssim  (1+|\lambda|)^{n-3+ \frac{2(p'-p)}{p'(2-p)}}.
    \end{align*}
    Next, we decompose the integral on the left-hand side of \eqref{eqn_HYP_P} with respect to $\lambda $ into regions where  $|\lambda|$ is near zero and away from zero. Then, plugging in the inequalities above and utilizing \eqref{eqn_HYP_nP},   \eqref{eqn_HYP_P} follows.
    \end{proof}
\begin{remark}
    We note that one can extend an analogue of the Hausdorff-Young inequality for non-unitary duals (Theorem \ref{thm_H-Y_int})  to higher-rank symmetric spaces. However, for the Paley inequality (Theorem \ref{thm_paley_int}), we rely on both upper and lower estimates of $|c(\lambda)|^{-2}$, which are available only in the rank one case \eqref{sharp_c-2}. We plan to revisit this problem in the near future.
    \end{remark}

\section{Harmonic $NA$ groups, revisited}\label{sec_NA_group} 
In this section, we provide a brief overview of how the results in this article can be extended to the setting of harmonic $NA$ groups, also known as Damek-Ricci spaces. These groups were introduced by Damek and Ricci as a family of counterexamples to the Lichnerowicz conjecture in the noncompact case \cite{DR_92}. A harmonic $NA$ group $S$ is a semidirect product $N\rtimes A$, where $N$ is a $H$-type group, $A=(0,\infty)$, and the action of $A$ on $N$ is anisotropic dilations. For more details, we refer to \cite{ACD97, ADY96}. 

In what follows, $S=NA$ is a harmonic $NA$ group equipped with a left Haar measure $dx$. We recall from \cite{ACD97} that in this context, there exists an analogue of the (Helgason) Fourier transform $\widetilde{f}$ on symmetric spaces realized in the noncompact picture, where the role of the sphere is played by the noncompact subgroup $N$. Using this Fourier transform, for a given bounded measurable function $m$ on $\R$, we can define the Fourier multiplier operator $T_m$ as follows
\begin{align}
    \widetilde{T_m f}(\lambda, n) = m(\lambda) \widetilde{f}(\lambda, n), \qquad \lambda \in \R, n\in N,
\end{align}
for all $f\in C_c^{\infty}(S)$. 
Following the works \cite{CGHM94, GMM97}, Astengo studied the $L^p$-boundedness problems of multipliers in \cite{As95a, As95b} in this setting. Anker et al. \cite{ACD97} obtained sharp and simple criteria for the $L^p \rightarrow L^p$ and the weak $L^1 \rightarrow L^1$  boundedness of positive convolution kernels, which illustrated the similarity with symmetric spaces. Later, several authors studied analogous problems on spectral multipliers in these groups; for further details, see \cite{HMM05, Va07} and the reference within.  Building on the aforementioned results, we establish the following criteria for the  $L^p \rightarrow L^q$  boundedness of Fourier multipliers on harmonic $NA$ groups.
\begin{theorem}\label{thm_Lp-Lq_NA}
     Let $1<p\leq 2 \leq q<\infty$. Then there exists a constant $C_{p,q} > 0$, such that the following holds
     \begin{align*}
          \|T_m f\|_{L^q(S)}& \leq C_{p,q} \left( \sup\limits_{\alpha>0}\,\alpha \left( \int\limits_{\{ \lambda \in \R : |m(\lambda)|>\alpha \}   }  |c(\lambda)|^{-2} d\lambda\right)^{{\frac{1}{p}-\frac{1}{q}}} \right) \|f\|_{L^p(S)}
     \end{align*}
     for all $f \in L^p(S)$, where $|c(\lambda)|^{-2} \, d\lambda$ is the Plancherel measure on $S$. 
\end{theorem}
The proof of the theorem above follows similarly to Theorem \ref{thm_Lp-Lq_int}, so we will briefly outline the key steps. Ray and Sarkar \cite{RS_09} originally developed the theory of restriction theorems and proved the Hausdorff-Young inequality in the context of harmonic $NA$ groups. We can utilize their results and the calculations presented in Section \ref{sec_HYP_UD} of this article to prove  Theorem \ref{thm_Lp-Lq_NA}, with the nilpotent group $N$ playing the role of $K$. Specifically, we already have the required Hausdorff-Young inequality from \cite[Theorem 4.2]{RS_09}. To prove the Paley-type inequality, we define the following sublinear operator (see  \eqref{def_T_pal})
\begin{align*}
\mathcal{T}f(\lambda) = \frac{1}{u(\lambda)}\|\widetilde{f}(\lambda,\cdot)\|_{L^2(N)}, \quad \lambda \in \R.
\end{align*}
Then, the proof for the Paley-type inequality follows similarly to that of Theorem \ref{thm_paley_uni}. Consequently, using interpolation theory, we can prove a similar Hausdorff-Young-Paley inequality as in Theorem \ref{thm_HYP_uni}. Once we have this inequality, we can follow the technique used in the proof of Theorem \ref{thm_Lp-Lq_int} with obvious modifications.

\begin{remark}
\begin{enumerate}
    \item The other results in this article can be similarly extended to harmonic $NA$ groups by replacing the role of $K$ with $N$. In particular, all the results in Section \ref{subs_App_Lp_Lq} are true in this setting.  Since the methods are similar, we omit the detailed proofs. 
    \item We would like to mention that Kumar and Ruzhansky \cite[Theorem 5.6]{KR21} proved an analogue of Theorem \ref{thm_Lp-Lq_NA} for the Jacobi transform, particularly on harmonic $NA$ groups in the radial setting.
    \item In \cite{AR20} Akylzhanov and Ruzhansky proved the $L^p \rightarrow L^q$ boundedness for general locally compact separable unimodular groups. Our results extend this to harmonic $NA$ groups, which are non-unimodular. To the best of our knowledge, these H\"ormander type $L^p \rightarrow L^q$ results on harmonic $NA$ groups are the first of their kind and open the possibility for extending this investigation to a broader context of non-unimodular groups.
\end{enumerate}
\end{remark}

   \section{Theory of Stein-Weiss analytic interpolation}\label{sec_SW_analytic}
    In this section, we will develop the theory of Stein-Weiss analytic interpolation for mixed norm spaces. Let $(X, dx)$ and $(Y,dy)$ be two $\sigma$-finite measure spaces. Throughout this section, we will assume that all linear operators $T$ under consideration adhere to the following condition that $ \int_Y T f(y) g(y) dy$ is well-defined for all simple functions $f$ and $g$ with finite measure support. Then, the following theorem is well known; see {\cite[Theorem 1]{Stein_56}}.  
    \begin{theorem}\label{thm_steinlytic}
        Let $1\leq  p_0,p_1, q_0,q_1\leq  \infty$ and let $T_z $ be an analytic family of linear operators between $X$ and $Y$ of \textit{admissible growth} (in the sense of \cite[(1.1)]{Stein_56}), defined in the strip $\{ z \in \C : 0 \leq  \Re z \leq 1 \}$. Suppose that for any simple function $f$ with finite measure support on $X$ and $g$ on $Y$, the expression $ \int_Y T_z f(y) g(y) dy$
    is absolutely convergent and forms a complex analytic function in $z$ on the strip. Assume that we have the strong-type bounds 
        \begin{align*}
         \|T_{j+i\xi}(f)\|_{L^{q_j} (Y,dy)} \leq A_j(\xi) \|f\|_{L^{p_j}(X, dx)}
        \end{align*}
        for all simple functions $f$ of finite measure support, where $j = 0, 1$ and $\log |A_j(\xi)| \leq C e^{a|\xi|}$, with $a<\pi$. Then we have 
             \begin{align*}
            \| T_\theta f \|_{L^{q_\theta}(Y, dy)} \leq A_\theta \| f\|_{L^{p_\theta} (X, dx)}
        \end{align*}
        for all simple functions $f$ of finite support and $0\leq \theta \leq 1$, where
        \begin{equation}\label{p_theta_exp}
    \begin{aligned}
       &\frac{1}{p_\theta}=  \frac{1-\theta}{p_0}+\frac{\theta}{p_1}, \qquad \quad \quad \frac{1}{q_{\theta}}= \frac{1-\theta}{q_0}+\frac{\theta}{q_1}.
       \end{aligned}
       \end{equation}
    \end{theorem}

Now, we present the following theorem, which could be viewed as an analytical counterpart to the Stein-Weiss interpolation theorem; see \cite[page 120]{BL76}.
\begin{theorem}[Stein-Weiss analytic interpolation]\label{thm_SW_an_1d} 
Let $1\leq p_0, p_1, q_0,q_1 \leq \infty$, and let $T_z $ be an analytic family of linear operators between $X$ and $Y$ of \textit{admissible growth}, defined in the strip $\{ z \in \C : 0 \leq  \Re z \leq 1 \}$. Suppose that $v_0, v_1: X \rightarrow \R^+$ and ${w}_0, {w}_1:Y \rightarrow \R^+$ are non-negative weights integrable on every finite measure set.  We further assume that for all finite linear combinations of characteristic functions of rectangles of finite measures $f$ on $X$:
\begin{align}\label{hyp_T_z}
\|T_{j+i\xi}(f)\|_{L^{q_j} (Y, {w}_j dy)} \leq A_j(\xi) \|f\|_{L^{p_j}(X,v_j dx)}
\end{align}
where $j = 0, 1$ and $\log |A_j(\xi)| \leq C e^{a|\xi|}$, with $a<\pi$. Then we have 
     \begin{align*}
        \|T_{\theta}(f)\|_{L^{q_\theta} (Y, {w}_\theta dy)} \leq A_{\theta} \|f\|_{L^{p_\theta}(X,v_{\theta} dx)}
    \end{align*}
    for all simple functions $f$ of finite measure support and $0\leq  \theta \leq 1$, where $p_\theta$, $q_{\theta}$, $w_{\theta}$, and $\tilde{w}_{\theta}$ are as follows
     \begin{equation}\label{theta_exp}
    \begin{aligned}
       &\frac{1}{p_\theta}=  \frac{1-\theta}{p_0}+\frac{\theta}{p_1}, \qquad \quad \quad &&\frac{1}{q_{\theta}}= \frac{1-\theta}{q_0}+\frac{\theta}{q_1},\\
       & v_{\theta} = v_0^{\frac{p_\theta (1-\theta)}{p_0}} v_1^{\frac{p_{\theta} \theta}{p_1}},  \qquad \quad    &&{w}_{\theta} = {w}_0^{\frac{q_\theta (1-\theta)}{q_0}} {w}_1^{\frac{q_{\theta} \theta}{q_1}}.
       \end{aligned}
       \end{equation}
\end{theorem}
\begin{proof}
Following a similar argument by Tao \cite{Tao_n}, we can make qualitative reductions, where we can assume the weights $v_j$ and ${w}_j$ are bounded both above and below. We can approximate them by simple functions and employ a limiting argument to reduce to the case where $v_j$, ${w}_j$ are simple and never vanish. Now, we define the following analytic family of linear operators
\begin{align*}
    \mathcal{S}_z(f) = v_0^{\frac{1-z}{q_0}}v_1^{\frac{z}{q_1}}T_z\left(f {w}_0^{-\frac{(1-z)}{p_0}}  {w}_1^{-\frac{z}{p_1}}\right)
\end{align*}
for all $z \in \C$ with $0\leq \Re z\leq 1$, which is well-defined for all simple functions $f$ with finite measure support. The hypotheses \eqref{hyp_T_z} then ensures that 
\begin{align*}
\|\mathcal{S}_{j+i\xi}(f)\|_{L^{q_j} (Y, dy)} \leq A_j(\xi) \|f\|_{L^{p_j}(X,dx)}
\end{align*}
for $j=0,1$. Thus, applying the Stein analytic interpolation theorem, we obtain
\begin{align}\label{bdd_S_8}
    \|\mathcal{S}_{\theta}(f)\|_{L^{q_\theta} (Y, dy)} \leq A_\theta \|f\|_{L^{p_\theta}(X,dx)},
\end{align}
for all $0\leq \theta\leq 1$. Now, writing out the definition of the operator $\mathcal{S}_\theta$ in \eqref{bdd_S_8}, we get 
\begin{align*}
     \left( \int_Y v_0^{q_\theta \frac{1-\theta}{q_0}}v_1^{q_{\theta}\frac{\theta}{q_1}}|T_\theta(f {w}_0^{-\frac{(1-\theta)}{p_0}}  {w}_1^{-\frac{\theta}{p_1}})|^{q_\theta} dy\right)^{\frac{1}{q_{\theta}}}\leq A_{\theta} \left( \int_X |f|^{p_{\theta}} dx\right)^{\frac{1}{p_{\theta}}}
\end{align*}
for any simple functions $f$ on $X$ with finite measure support, which implies
\begin{align*}
     \left( \int_Y |T_\theta(f)|^{q_\theta} v_0^{q_\theta \frac{1-\theta}{q_0}}v_1^{q_{\theta}\frac{\theta}{q_1}} dy\right)^{\frac{1}{q_{\theta}}}\leq A_{\theta} \left( \int_X  |f|^{p_{\theta}} {w}_0^{ p_{\theta}\frac{(1-\theta)}{p_0}}  {w}_1^{p_{\theta}\frac{\theta}{p_1}} dx\right)^{\frac{1}{p_{\theta}}}
\end{align*}
for all simple functions with finite measure support.
This concludes the proof of our theorem.
\end{proof}
We will need a mixed norm space version of the aforementioned theorem, which can be derived similarly to \eqref{eqn_ana_int} using Theorem \ref{thm_SW_an_1d}. To simplify notation, we will state a specific version of Theorem \ref{thm_SW_an_1d} that is applicable to mixed norm spaces and relevant to this study.

Let us consider two product measure spaces $X=X_0 \times X_1$ and $Y=Y_0 \times Y_1$,  with $d x$ and $d y$ denoting the respective (product) measures on $X$ and $Y$. Then, we have the following corollary of Theorem \ref{thm_SW_an_1d}  on mixed norm spaces (see Section \ref{subsec: mix_norm} for definition of mixed norm spaces).
\begin{corollary}\label{cor_sw_ana} 
Let $1\leq p, q_j,\tilde{q_j}\leq \infty$ and $T_z $ be an analytic family of linear operators between $X$ and $Y$ of \textit{admissible growth}, defined in the strip $\{ z \in \C : 0 \leq  \Re z \leq 1 \}$. Suppose that ${w}_j: Y_0 \rightarrow \R^+$ are non-negative weights integrable on every finite measure set for $j=0,1$. We further assume that for all finite linear combinations of characteristic functions of rectangles of finite measures $f$ on $X$:
\begin{align}\label{hyp_T_z_SW}
\|T_{j+i\xi}(f)\|_{L^{Q_j} (Y, { w}_j dy)} \leq A_j(\xi) \|f\|_{L^{P}(X,dx)}
\end{align}
where $j = 0, 1$, \begin{align*}
    Q_0 = (q_0,\tilde{q_0}),\qquad Q_1 = (q_1,\tilde{q_1} ), \qquad P=(p,p),
\end{align*} 
and $\log |A_j(\xi)| \leq C e^{a|\xi|}$, with $a<\pi$. Then we have 
     \begin{align*}
        \|T_{\theta}(f)\|_{L^{B} (Y, {w}_{\theta} dy)} \leq A_{\theta} \|f\|_{L^{P}(X, dx)}
    \end{align*}
    for all simple functions $f$ of finite measure support, $0<\theta<1$, where $B= (\beta,b)$ with 
    \begin{align}
        \frac{1}{\beta}= \frac{1-\theta}{q_0}+\frac{\theta}{{q_1}},  \qquad \frac{1}{b}= \frac{1-\theta}{\tilde{q_0}}+\frac{\theta}{\tilde{q_1}}, \qquad {w}_{\theta}= {w}_0^{\frac{b(1-\theta)}{\tilde{q_0}}} {w}_1^{\frac{b\theta}{\tilde{q_1}}}.
    \end{align}
\end{corollary}

\section*{Final remarks}
 In conclusion, we would like to draw attention to a few open questions that, from our perspective, warrant further investigation.
 \begin{enumerate}
     \item We will revisit this $L^p \rightarrow L^q$ Fourier multiplier problem in the near future to extend our results to reducible symmetric spaces of the form $\X_1 \times \cdots \times \X_m$, where $m \geq 2$ and each $\X_i$ is a Riemannian symmetric space.
     \item Additionally, it would be interesting to investigate whether an analogue of the $L^p \rightarrow L^q$ boundedness of Fourier multiplier theorem can be established in locally symmetric spaces for $1<p\leq 2 \leq q <\infty$.
 \end{enumerate}
\section*{Acknowledgments}
The authors thank Sanjoy Pusti for several useful discussions during the course of this work. TR also thanks Vishvesh Kumar and MR for making him interested in $L^p \rightarrow L^q$ multiplier problems.   TR and MR are supported by the FWO Odysseus 1 grant G.0H94.18N: Analysis and Partial Differential Equations, the Methusalem program of the Ghent University Special Research Fund (BOF), (TR Project title: BOFMET2021000601). TR is also supported by BOF postdoctoral fellowship at Ghent
University BOF24/PDO/025.  MR is also supported by EPSRC grant EP/V005529/7.
\bibliographystyle{alphaurl}
\bibliography{Ref_HYP}
\end{document}